\documentclass[11pt]{amsart}

\usepackage{amsmath}
\usepackage{amssymb}
\usepackage{latexsym}
\usepackage{amscd}
\usepackage{mathrsfs}
\usepackage[all]{xy}

\usepackage{hyperref}

\usepackage{color}

\newdimen\AAdi%
\newbox\AAbo%
\def\AAk#1#2{\s_etbox\AAbo=\hbox{#2}\AAdi=\wd\AAbo\kern#1\AAdi{}}%
\def\AAr#1#2#3{\s_etbox\AAbo=\hbox{#2}\AAdi=\ht\AAbo\raise#1\AAdi\hbox{#3}}%
\font\tenmsb=msbm10 at 12pt \font\sevenmsb=msbm7 at 8pt
\font\fivemsb=msbm5 at 6pt
\newfam\msbfam
\textfont\msbfam=\tenmsb \scriptfont\msbfam=\sevenmsb
\scriptscriptfont\msbfam=\fivemsb

\newtheorem{theorem}{Theorem}

\newtheorem{remark}[theorem]{Remark}

\newtheorem{corollary}[theorem]{Corollary}

\newtheorem{lemma}[theorem]{Lemma}

\numberwithin{equation}{section} \numberwithin{theorem}{section}

\renewcommand{\topmargin}{0cm}
\renewcommand{\oddsidemargin}{5mm}
\renewcommand{\evensidemargin}{5mm}
\renewcommand{\textwidth}{150mm}
\renewcommand{\textheight}{230mm}

\def\R{\mathbb R}

\def\S{\mathbb S}

\def\na{\nabla}
\def\bn{\overline\nabla}

\def\ir#1{\mathbb R^{#1}}

\def\f#1#2{\frac{#1}{#2}}

\def\a{\alpha}
\def\be{\beta}

\def\r{\Re_{I\!V}}

\def\p#1{\partial #1}

\def\de{\delta}
\def\De{\Delta}
\def\e{\eta}
\def\ep{\epsilon}

\def\G{\Gamma}
\def\g{\gamma}
\def\k{\kappa}
\def\la{\lambda}

\def\lan{\langle}
\def\ran{\rangle}

\def\Om{\Omega}
\def\th{\theta}

\def\si{\sigma}
\def\Si{\Sigma}

\def\r{\rho}
\def\z{\zeta}
\def\div{\mathrm{div}}

\begin{document}

\title
{Minimal graphic functions on manifolds of non-negative Ricci curvature}
\author{Qi Ding}
\address{Max Planck Institute for Mathematics in the Sciences, Inselstr. 22, 04103 Leipzig, Germany}
\email{qiding@mis.mpg.de}\email{09110180013@fudan.edu.cn}
\author{J. Jost}
\address{Max Planck Institute for Mathematics in the Sciences, Inselstr. 22, 04103 Leipzig, Germany}
\email{jost@mis.mpg.de}
\author{Y.L. Xin}
\address{Institute of Mathematics, Fudan University, Shanghai 200433, China}
\email{ylxin@fudan.edu.cn}

\thanks{The second author is supported by the ERC Advanced Grant
  FP7-267087. The  third  author are supported partially by NSFC. He is also grateful to the Max Planck
Institute for Mathematics in the Sciences in Leipzig for its
hospitality and  continuous support. }
\date{}
\begin{abstract}
We study minimal graphic functions on   complete Riemannian manifolds
$\Si$ with non-negative Ricci curvature, Euclidean volume growth and
quadratic curvature decay. We derive global bounds for the gradients
for minimal graphic functions  of linear growth only on one side.
Then we can obtain a Liouville type theorem with such growth via splitting for tangent cones of $\Si$ at infinity.
When, in contrast, we do not impose  any growth restrictions for minimal
graphic functions, we also obtain a Liouville type theorem under a certain
non-radial Ricci curvature decay  condition on $\Si$. In particular, the borderline for the Ricci curvature decay is sharp by our example in the last section.
\end{abstract}

\maketitle

\section{Introduction}
The minimal surface equation on a Euclidean
space,
\begin{equation}\label{u00}
\div \left(\f{Du}{\sqrt{1+|Du|^2}}\right)=0
\end{equation}
has been investigated extensively by many mathematicians. Let
us recall some famous results that constitute a background for our present work.  In 1961, J. Moser \cite{M} derived
Harnack inequalities for uniformly elliptic equations that imply
weak Bernstein results for  minimal graphs in any dimension.  In 1969,
Bombieri-De Giorgi-Miranda \cite{BDM} showed interior gradient
estimates for solutions to the minimal surface equation (see also the
exposition in chapter 16 of \cite{GT});    the two-dimensional case
had already been obtained by Finn \cite{F} in 1954.

In this paper, we study the non-linear partial differential equation
describing minimal graphs over complete Riemannian manifolds of
non-negative Ricci curvature. Formally, the equation for a minimal
graph over a Riemannian manifold  $\Si$ is the same as for
Euclidean
space,
\begin{equation}\label{u0}
\div_\Si\left(\f{Du}{\sqrt{1+|Du|^2}}\right)=0,
\end{equation}
where the divergence operator and the norm now are defined in terms of the Riemannian metric of $\Si$.

The geometric content of  (\ref{u0}) is that a  solution $u$ is
the height function in the product manifold
$N=\Si\times \ir{}$ of a minimal graph $M$ in $N$. We therefore call a
solution to (\ref{u0}) a {\it a minimal graphic function}.

Of course, the Riemannian equation \eqref{u0} is more difficult than
its Euclidean version \eqref{u00}. In order to obtain strong results,
one needs to restrict the class of underlying Riemannian
manifolds. The linear analogue, the equation for harmonic functions on
Riemannian manifolds, suggests that non-negative Ricci curvature
should be a good geometric condition. In fact, harmonic functions on complete manifolds with non-negative Ricci
curvature have been very successfully  studied by S. T. Yau \cite{Y}, Colding-Minicozzi
\cite{CM2}, P. Li \cite{Li} and many others. Our problem can be
considered as  a non-linear generalization of harmonic functions on complete manifolds with non-negative Ricci curvature.

More precisely, we consider   complete non-compact $n-$dimensional
Riemannian manifolds $\Si$  satisfying the three  conditions:

\noindent C1) Non-negative Ricci curvature;\\
\noindent C2) Euclidean volume growth;\\
\noindent C3) Quadratic decay of the curvature tensor.

Fischer-Colbrie and Schoen \cite{FS} studied stable minimal surfaces in 3-dimensional manifolds with nonnegative scalar curvature, and showed their rigidity.
In our companion paper \cite{DJX},  we study minimal hypersurfaces in $\Si$ and obtain existence and non-existence  results for area-minimizing
hypersurfaces in such an $\Si$. Here, we restrict the dimension $n+1\ge4$ for ambient product manifolds and investigate minimal graphs from the PDE
point of view.

 Cheeger and Colding \cite{CCo1,CCo2,CCo3} studied the structure of  pointed Gromov-Hausdorff limits  of sequences $\{M_i^n,p_i)\}$ of complete connected Riemannian manifolds with $Ric_{M_i^n}\ge-(n-1)$. They showed that the singular set $\mathcal{S}$ of such a space has dimension no bigger than $n-2$. Subsequently, Cheeger-Colding-Tian \cite{CCT}  showed the stronger statement that $\mathcal{S}$ has dimension no bigger than $n-4$ under some additional assumption. We should point out that conditions C1) --C3) still permit some
nasty behavior of the manifold $\Si$. For instance, as G. Perelman pointed out, the tangent cones at infinity need not be unique.
Our conditions C2) and C3) also have appeared in the investigation of the uniqueness of tangent cones at infinity for Ricci flat manifolds by Cheeger and Tian \cite{CT}. This theory has recently been further developed by Colding and Minicozzi \cite{CM3}.

In the last decade, minimal graphs in product manifolds received
considerable attention. Concerning gradient estimates, J. Spruck \cite{Sp} obtained interior gradient estimates
via the  maximum principle. He went on to apply them  to the Dirichlet problem for constant mean curvature graphs.
Recently,  Rosenberg-Schulze-Spruck \cite{RSS} obtained a new gradient estimate and then showed that there is no trivial positive entire minimal graph over any manifold with nonnegative Ricci curvature and curvature bounded below.

For a complete  manifold $\Si$ with conditions C1), C2) and C3) we
obtain the  gradient estimates for minimal graphic functions by
integral methods, see Theorem \ref{GEu}. Such results in Euclidean
space were given by \cite{BDM},  \cite{BG}. Our results and methods
are different from those in \cite{Sp}.

Let us now describe our results in more precise PDE terms.
Theorem \ref{GEu} enables us to obtain  global bounds for gradients
when the growth for the minimal graphic
functions is linearly constrained only  on one side. So linear growth is equivalent to linear growth on one side for minimal graphic functions.
Since the Sobolev inequality and the Neumann-Poincar$\acute{\mathrm{e}}$ inequality both
hold on $\Si$, and thus  also  on the minimal graph $M$ in
$\Si\times\R$ represented by a minimal graphic function with bounded
gradient, this will then yield mean value inequalities for subharmonic
functions on the domains of $M$. Hence, we have the mean value equalities both in the
extrinsic balls or the intrinsic balls of $M$. Therefore, we obtain a
Liouville type theorem for minimal graphic functions with sub-linear
growth on one side, see Theorem \ref{Liusub}. It is interesting to
compare this Liouville  type theorem with the half-space theorem of Rosenberg-Schulze-Spruck \cite{RSS}.
The corresponding result for harmonic functions on manifolds with non-negative Ricci  curvature is due to S. Y. Cheng \cite{Ch}.
Furthermore, we can relax sub-linear growth to linear growth in the above Liouville type theorem if $\Si$ is not Euclidean space, and obtain the following theorem.
\begin{theorem}\label{IntrL}
Let $u$ be an entire solution to \eqref{u} on a complete Riemannian manifold $\Si$ with conditions C1), C2), C3).
If $u$ has at most linear growth on one side, then $u$ must be a constant unless $\Si$ is isometric to Euclidean space.
\end{theorem}
For showing Theorem \ref{IntrL}, using the harmonic coordinates of
Jost-Karcher \cite{JK}, we first  obtain
scale-invariant Schauder estimates for minimal graphic functions
$u$. Then combining this with estimates for the  Green function and a
Bochner type formula, we get integral decay estimates for the Hessian
of $u$. The Schauder estimates then imply point-wise decay estimates
for the Hessian of $u$. Finally, by re-scaling the  manifold $\Si$ and
the graphic function $u$ we can show that $\Si$ is asymptotic to a
product of a Euclidean factor $\R$ and the level set of $u$ at the
value $0$. This will allow us to deduce Theorem \ref{IntrL}. By the
example in the last section, the assumption of linear growth cannot be removed. Moreover, one should  compare this theorem  with the result of
Cheeger-Colding-Minicozzi \cite{CCM} on the splitting  of the tangent cone at infinity for
complete manifolds with nonnegative Ricci curvature supporting linear growth harmonic functions.

We shall also investigate minimal graphic functions
without growth restrictions. Analogously to \cite{DJX} we prove that
for any scaling sequence of a minimal graph $M$ in $\Si\times\R$ there
exists a subsequence  that  converges to an area-minimizing cone
$T$ in $\Si_\infty\times\R$, where $\Si_\infty$ is some tangent cone
at infinity (not necessarily unique) of $\Si$ satisfying conditions
C1), C2) and C3). However, our proof here is more complicated than \cite{DJX} as the ambient manifold $\Si\times\R$ does no longer satisfy condition C3) unless $\Si$ is flat. In the light of stability inequalities for minimal
graphs, one might expect that $T$ is vertical, namely,
$T=\mathcal{X}\times\R$ for some cone $\mathcal{X}\in
N_\infty$. However, since we lack Sobolev  and
Neumann-Poincar$\acute{\mathrm{e}}$ inequalities on minimal graphs
whose gradient is not globally bounded, it seems difficult to show
verticality, although such  inequalities hold for ambient Euclidean
space \cite{MS}, \cite{BG}. Fortunately, we are able to show that $T$ is asymptotically
vertical at infinity. This  also suffices to estimate the measure of a 'bad' set and employ stability arguments,
as  in \cite{DJX}, to eliminate the unbounded situation of $|D u|,$
when the lower bound $\k$ for the curvature decay (see below) is large, that is, when the curvature can only slowly decay to 0 at infinity (more precisely, the curvature decay is quadratic, but the value of the lower bound must be above some critical threshold; actually that threshold is sharp showed in the last section).
Thus combining Theorem \ref{IntrL} we  obtain a
Liouville type theorem for minimal graphic  functions.

\begin{theorem}\label{BMGP1}
Let $\Si$ be a  complete $n-$dimensional Riemannian manifold
satisfying conditions C1), C2), C3) and with non-radial Ricci
curvature satisfying $\inf_{\p
  B_\r}Ric_{\Si}\left(\xi^T,\xi^T\right)\ge\k\r^{-2}|\xi^T|^2$ for
some constant $\k$, for sufficiently large $\r>0$, where $\r$ is the distance function from a fixed
point in $\Si$ and $\xi^T$ stands for the  part that is tangential to
the geodesic sphere $\p B_\r$ (at least away from the cut locus of the center),  of a tangent vector $\xi$ of $\Si$ at the considered point. If $\k>\f{(n-3)^2}4$, then any entire solution to \eqref{u0} on $\Si$ must be  constant.
\end{theorem}

In the last section, we construct a nontrivial minimal graph to show
that the constant $\f{(n-3)^2}4$ in Theorem \ref{BMGP1} is sharp. Our
approach is inspired by
 the  method developed  by Simon \cite{Si2}, where for each strictly
 minimizing isoparametric cone $C$ in $\R^n$ he constructed an entire
 minimal graph in $\R^{n+1}$ converging to $C\times\R$ as tangent
 cylinder at $\infty$. Geometrically, our method is different from
 that of Simon, but analytically, it is quite similar.

\section{Preliminaries}

Let $(\Si,\si)$ be an $n$-dimensional complete non-compact manifold
with Riemannian metric $\si=\sum_{i,j=1}^n\si_{ij}dx_idx_j$ in a local
coordinate. Set $(\si^{ij})$ be the inverse matrix of $(\si_{ij})$ and
$E_i$ be the dual frame of $dx_i$. Denote
$Du=\sum_{i,j}\si^{ij}u_iE_j$ and
$|Du|^2=\sum_{i,j}\si^{ij}u_iu_j$. Let $\div_{\Si}$ be the divergence
of $\Si$. We shall study the following quasi-linear elliptic equation
for a minimal graphic function
on $\Si$
\begin{equation}\label{u}
\div_\Si\left(\f{Du}{\sqrt{1+|Du|^2}}\right)\triangleq\f1{\sqrt{\det \si_{kl}}}\p_j\left(\sqrt{\det \si_{kl}}\f{\si^{ij}u_i}{\sqrt{1+|Du|^2}}\right)=0.
\end{equation}
A solution  $\{(x,u(x))\in \Si\times\R|\ x\in\Si\}$ represents a minimal hypersurface in the product manifold $N\triangleq\Si\times\R$ with the product metric
$$ds^2=dt^2+\si=dt^2+\sum_{i,j}\si_{ij}dx_idx_j.$$
Let $M$ denote this hypersurface, i.e.,
$M=\{(x,u(x))\in N|\ x\in\Si\},$
with the induced metric $g$ from $N$
$$g=\sum_{i,j}g_{ij}dx_idx_j=\sum_{i,j}(\si_{ij}+u_iu_j)dx_idx_j,$$
where $u_i=\f{\p u}{\p x_j}$ and and $u_{ij}=\f{\p^2u}{\p x_i\p x_j}$ in the sequel.
Moreover, $\det g_{ij}=(1+|Du|^2)\det \si_{ij}$ and $g^{ij}=\si^{ij}-\f{u^iu^j}{1+|Du|^2}$ with $u^i=\sum_j\si^{ij}u_j$.

Let $\De$, $\na$ be the Laplacian operator and Levi-Civita connection
of $(M,g)$, respectively. The equation \eqref{u} then becomes
\begin{equation}\aligned
\De u=\f1{\sqrt{\det g_{kl}}}\p_j\left(\sqrt{\det g_{kl}}g^{ij}u_i\right)=0.
\endaligned
\end{equation}
Thus, $u$ is a harmonic function on the hypersurface$M$, which in turn
depends on $u$.

Similar to the Euclidean case (\cite{X} or by Lemma 2.1 in
\cite{DJX}), any minimal graph over a bounded domain $\Om$ is an area-minimizing hypersurface in $\Om\times\R$. From the proof of Lemma 2.1 in
\cite{DJX}, it is not hard to see that any minimal graph over $\Si$ is an area-minimizing hypersurface in $\Si\times\R$.

Let $\overline{\na}$ and $\overline{R}$ be the Levi-Civita connection and curvature tensor of $(N,h)$. Let $\lan\cdot,\cdot\ran$ be the inner product on $N$ with respect to its metric. Let $\nu$ denote the unit normal vector field of $M$ in $N$ defined by
\begin{equation}\label{No}
\nu=\f1{\sqrt{1+|Du|^2}}(-Du+E_{n+1}).
\end{equation}
Choose a local orthonormal frame field $\{e_i\}_{i=1}^n$ in $M$. Set the coefficients of the second fundamental form $h _{ij}=\lan\overline{\na}_{e_i}e_j,\nu\ran$
and the squared norm of the second fundamental form $|B|^2=\sum_{i,j}h_{ij}h_{ij}$. Then the mean curvature $H=\sum_ih_{ii}=0$ as $M$ is minimal.

Let $\overline{Ric}$ be the Ricci curvature of $\Si\times\R$.
Due to the Codazzi equation $h_{ijk}=h_{kji}-\lan\overline{R}_{e_ke_i}e_j,\nu\ran$ (see \cite{X}, for example),
we obtain a Bochner type formula
\begin{equation}\aligned\label{Enu2}
\De\lan E_{n+1},\nu\ran
=&-\big(|B|^2+\overline{Ric}(\nu,\nu)\big)\lan E_{n+1},\nu\ran.
\endaligned
\end{equation}
Let $Ric$ denote the Ricci curvature of $\Si$. With (\ref{No}), then
\begin{equation}\aligned\label{Logv}
-\De\log\lan E_{n+1},\nu\ran=|B|^2+\f{Ric(Du,Du)}{1+|Du|^2}+|\na\log\lan E_{n+1},\nu\ran|^2.
\endaligned
\end{equation}
The above formula will play a similar role as Bochner's formula for
the squared length of the gradient of a harmonic function.

In the present paper we usually suppose that $\Si$ is an $n(\ge3)$-dimensional complete non-compact Riemannian manifold satisfying the following three conditions:

C1) Nonnegative Ricci curvature: Ric $\ge 0$;

C2) Euclidean volume growth: for the geodesic balls $B_r(x)$ in $\Si$,
\begin{equation*}
V_\Si\triangleq\lim_{r\rightarrow\infty}\f{Vol\big(B_r(x)\big)}{r^n}>0;
\end{equation*}

C3) Quadratic  decay of the curvature tensor: for sufficiently large $\r(x)=d(x,p)$, the distance from a fixed point in $N$,
\begin{equation*}
|R(x)|\le\f c{\r^{2}(x)}.
\end{equation*}

By the Bishop-Gromov volume comparison theorem
$\lim_{r\rightarrow\infty}\f{Vol(B_r(x))}{r^n}$ is monotonically
nonincreasing, and hence
\begin{equation}\aligned\label{VolSi}
V_\Si r^n\le Vol\big(B_r(x)\big)\le\omega_nr^n \qquad \mathrm{for\ all}\ x\in\Si\ \mathrm{and}\ r>0,
\endaligned
\end{equation}
where $\omega_n$ is the volume of the standard $n$-dimensional unit ball in Euclidean space.
The above three conditions have been used in \cite{DJX} to study minimal hypersurfaces in such manifolds.
Now we list some properties of $\Si$ satisfying conditions C1), C2), C3) (see \cite{DJX} for completeness), which will be employed in the following text.

\begin{itemize}
  \item By \cite{CGT}, there is a sufficiently small constant $\de_0>0$ depending only on $n,c,V_{\Si}$ so that for any $0<\de<\de_0$ the injectivity radius at $q\in\p B_{\Om r}(p)$ satisfies $i(q)\ge r$, where $\Om=\left(\f{\sqrt{c}}{\de}+1\right)$.
  \item Let $G(p,\cdot)$ be the Green function on $\Si$ with $\lim_{r\rightarrow0}\sup_{\p B_r(p)}\left|Gr^{n-2}-1\right|=0$ and $b=G^{\f1{2-n}}$(see \cite{MSY} and \cite{CM1} for details). Then
\begin{equation}\aligned\label{Deb}
\De_{\Si}b^2=2n|D b|^2
\endaligned
\end{equation}
with $|Db|\le1$ and $c'r\le b(x)\le r$ for any $n\ge3$, $x\in\p B_r(p)$ and some constant $c'>0$. Moreover, we have asymptotic estimates
\begin{equation}\aligned\label{br}
\limsup_{r\rightarrow\infty}\left(\sup_{x\in\p B_r}\left(\left|\f br-\left(\f{V_{\Si}}{\omega_n}\right)^{\f1{n-2}}\right|+\left|Db-\left(\f{V_{\Si}}{\omega_n}\right)^{\f1{n-2}}\right|\right)\right)=0,
\endaligned
\end{equation}
and
\begin{equation}\aligned\label{Hessb2}
\limsup_{r\rightarrow\infty}\left(\sup_{x\in\p B_r}\left|\mathrm{Hess}_{b^2}-2\left(\f{V_{\Si}}{\omega_n}\right)^{\f2{n-2}}\si\right|\right)=0.
\endaligned
\end{equation}
\item By Gromov's compactness theorem \cite{GLP} and Cheeger-Colding's theory \cite{CC}, for any sequence $\bar{\ep}_i\rightarrow0$
there is a subsequence $\{\ep_i\}$ converging to zero such that $\ep_i \Si=(\Si,\ep_i \si,p)$ converges to a metric cone $(\Si_{\infty},d_{\infty})$ with vertex $o$ in the pointed Gromov-Hausdorff sense. $\Si_\infty$ is called the tangent cone at infinity and \\ $\Si_{\infty}=CX=\R^+\times_{\r}X$ for some $(n-1)$-dimensional smooth compact manifold $X$ of Diam $(X)\le\pi$ and the metric $s_{ij}d\th_id\th_j$ with $s_{ij}\in C^{1,\a}(X)$ (see also \cite{GW}\cite{Ps}). For any compact domain $K\subset \Si_{\infty}\setminus\{o\}$, there exists a diffeomorphism $\Phi_i:\ K\rightarrow \Phi_i(K)\subset\ep_i\Si$ such that $\Phi_i^*(\ep_i \si)$ converges as $i\rightarrow\infty$ to $\si_{\infty}$ in the $C^{1,\a}$-topology on $K$.                                                                        \end{itemize}

\section{Gradient estimates and applications}

Let $u$ be a minimal graphic function on a Riemannian manifold $\Si$. For
$$v\triangleq\f1{\left\lan E_{n+1},\nu\right\ran}=\sqrt{1+|Du|^2},$$
we have derived  the Bochner type formula (\ref{Logv}).
Let $B_r(x)$ be the geodesic ball in $\Si$ with radius $r$ and centered at $x\in\Si$. Sometimes we write $B_r$ instead of $B_r(p)$ for simplicity. Let $d\mu$ be the volume element of $M$.
\begin{lemma}
Suppose $\Si$ has nonnegative Ricci curvature and $u(p)=0$, then for any constant $\be>0$
\begin{equation}\aligned\label{DuK}
\int_{B_r\cap\{|u|<\be r\}}\log v d\mu\le(1+10\be)\left(2+\be+r^{-1}\sup_{B_{3r}}u\right)Vol(B_{3r}).
\endaligned
\end{equation}
\end{lemma}
\begin{proof}
We define a function $u_s$ by
\begin{eqnarray*}
   u_s= \left\{\begin{array}{ccc}
           \be s  & \quad\quad  {\rm{if}} \ \ \     u\geq \be s \\ [3mm]
            u     & \quad\ \ \ \ {\rm{if}} \ \ \  |u|< \be s \\ [3mm]
           -\be s  & \quad\quad\ \ {\rm{if}} \ \ \   u\leq-\be s.
     \end{array}\right.
\end{eqnarray*}
$\r(x)$ is a global Lipschitz function with $|D\r|\equiv1$ almost everywhere. We define a Lipschitz function $\z(r)$ on $[0,\infty)$ satisfying $\mathrm{supp}\z\subset [0,2r]$, $\z\big|_{[0,r)}\equiv1$ and $|D\z|\le\f 1r$. Then $\eta(x)=\z(\r(x))$ is a Lipschitz function with $\mathrm{supp}\eta\subset \bar{B}_{2r}$, $\eta\big|_{B_r}\equiv1$ and $|D\eta|\le\f 1r$. Then by using (\ref{u}) and integrating by parts we have
\begin{equation}\aligned
0=-\int\div_\Si\left(\f{Du}{v}\right)\eta u_r d\mu_\Si\ge\int_{B_r\cap\{|u|<\be r\}}\f{|Du|^2}{v}d\mu_\Si+\int\f{Du\cdot D\eta}{v}u_rd\mu_\Si,
\endaligned
\end{equation}
which implies
\begin{equation}\aligned\label{VolDr}
\int_{B_r\cap\{|u|<\be r\}}1d\mu\le&\int_{B_r\cap\{|u|<\be r\}}\left(\f{|Du|^2}{v}+1\right)d\mu_\Si\\
\le& Vol(B_r)+\be r\int_{B_{2r}}|D\eta|d\mu_\Si\le (1+\be)Vol(B_{2r}).
\endaligned
\end{equation}
From integrating by parts, we deduce
\begin{equation}\aligned
0=&-\int\div_\Si\left(\f{Du}{v}\right)\cdot (u_r+\be r)\eta\log v\ d\mu_\Si\\
\ge&\int_{B_r\cap\{|u|<\be r\}}\f{|Du|^2}{v}\log v\ d\mu_\Si+\int\f{Du\cdot D\eta}{v}(u_r+\be r)\log v\ d\mu_\Si\\
&+\int\f{Du\cdot D\log v}{v}(u_r+\be r)\eta d\mu_\Si\\
\ge&\int_{B_r\cap\{|u|<\be r\}}\f{|Du|^2}{v}\log v\ d\mu_\Si-2\be r\int_{B_{2r}\cap\{u>-\be r\}}|D\eta|\log v\ d\mu_\Si\\
&-2\be r\int_{B_{2r}\cap\{u>-\be r\}}|D\log v|\eta d\mu_\Si,
\endaligned
\end{equation}
then we obtain
\begin{equation}\aligned\label{logv}
&\int_{B_r\cap\{|u|<\be r\}}\log v d\mu\le\int_{B_r\cap\{|u|<\be r\}}\left(\f{|Du|^2}{v}+1\right)\log v\ d\mu_\Si\\
\le&(1+2\be )\int_{B_{2r}\cap\{u>-\be r\}}\log v\ d\mu_\Si+2\be r\int_{B_{2r}\cap\{u>-\be r\}}|D\log v|\eta d\mu_\Si.
\endaligned
\end{equation}
Obviously, \eqref{Logv} implies $\De\log v\ge|\na\log v|^2$, then for any $\xi\in C^1_0(B_{2r}\times\R)$ we have
\begin{equation}\aligned
&\int|\na\log v|^2\xi^2d\mu\le\int\xi^2\De\log v d\mu=-2\int\xi\na\xi\cdot\na\log v d\mu\\
\le&\f12\int|\na\log v|^2\xi^2d\mu+2\int|\na\xi|^2d\mu.
\endaligned
\end{equation}
Set $\xi(x,t)=\eta(x)\tau(t)$ for $(x,t)\in\Si\times\R$, where $0\le\tau\le1$, $\tau\equiv1$ in $(-\be r,\sup_{B_{2r}}u)$, $\tau\equiv0$ outside $(-(1+\be)r,r+\sup_{B_{2r}}u)$, $|\tau'|<\f1{r}$. Then
\begin{equation}\aligned\label{nalogv}
\int|\na\log v|^2\xi^2d\mu\le&4\int|\na\xi|^2d\mu\le8\int\left(|\na\eta|^2\tau^2+|\na\tau|^2\eta^2\right)d\mu\\
\le&\f{16}{r^2}\int_{B_{2r}\cap\{u>-(1+\be)r\}}1d\mu.
\endaligned
\end{equation}
Together with
\begin{equation}\aligned
|\na\log v|^2=\sum_{i,j}g^{ij}\p_i\log v\cdot\p_j\log v=|D\log v|^2-\f{|Du\cdot D\log v|^2}{1+|Du|^2}\ge\f{|D\log v|^2}{1+|Du|^2},
\endaligned
\end{equation}
and \eqref{nalogv} it follows that
\begin{equation}\aligned\label{Dlogv}
&\int_{B_{2r}\cap\{u>-\be r\}}|D\log v|\eta d\mu_\Si\le\int_{B_{2r}\cap\{u>-\be r\}}|\na\log v|\eta vd\mu_\Si\\
\le&\int_{B_{2r}\cap\{u>-\be r\}}\left(\f{|\na\log v|^2\eta^2r}8+\f2{r}\right) vd\mu_\Si\\
\le&\f r8\int|\na\log v|^2\xi^2d\mu+\f2{r}\int_{B_{2r}\cap\{u>-\be r\}}vd\mu_\Si\\
\le&\f 2r\int_{B_{2r}\cap\{u>-(1+\be)r\}}1d\mu+\f2{r}\int_{B_{2r}\cap\{u>-\be r\}}vd\mu_\Si\\
\le&\f 4r\int_{B_{2r}\cap\{u>-(1+\be)r\}}vd\mu_\Si.
\endaligned
\end{equation}
Note that $\log v\le v$ as $v\ge1$, substituting \eqref{Dlogv} into \eqref{logv} we obtain
\begin{equation}\aligned\label{10be}
&\int_{B_r\cap\{|u|<\be r\}}\log v d\mu\le(1+2\be)\int_{B_{2r}\cap\{u>-\be r\}}v d\mu_\Si+8\be\int_{B_{2r}\cap\{u>-(1+\be)r\}}vd\mu_\Si\\
&\le(1+10\be)\int_{B_{2r}\cap\{u>-(1+\be)r\}}vd\mu_\Si.
\endaligned
\end{equation}
Let $\tilde{\eta}$ be a Lipschitz function with $\mathrm{supp}\tilde{\eta}\subset \bar{B}_{3r}$ with $\tilde{\eta}\big|_{B_{2r}}\equiv1$ and $|D\tilde{\eta}|\le\f 1r$. Then
\begin{equation}\aligned
0=&-\int_{B_{3r}}\div_\Si\left(\f{Du}{v}\right)\tilde{\eta}\cdot\max\{u+(1+\be)r,0\} d\mu_\Si\\
\ge&\int_{B_{2r}\cap\{u>-(1+\be)r\}}\f{|Du|^2}{v}d\mu_\Si+\int_{B_{3r}}\f{Du\cdot D\tilde{\eta}}{v}\max\{u+(1+\be)r,0\}d\mu_\Si\\
\ge&\int_{B_{2r}\cap\{u>-(1+\be)r\}}\f{|Du|^2}{v}d\mu_\Si-\f1r\int_{B_{3r}}\max\{u+(1+\be)r,0\}d\mu_\Si,
\endaligned
\end{equation}
which implies
\begin{equation}\aligned\label{1+ber}
\int_{B_{2r}\cap\{u>-(1+\be)r\}}vd\mu_\Si\le&\int_{B_{2r}\cap\{u>-(1+\be)r\}}\left(\f{|Du|^2}{v}+1\right)d\mu_\Si\\
\le&\int_{B_{2r}}d\mu_\Si+\f1r\int_{B_{3r}}\left(\sup_{B_{3r}}u+(1+\be)r\right)d\mu_\Si\\
\le&\left(2+\be+r^{-1}\sup_{B_{3r}}u\right)Vol(B_{3r}).
\endaligned
\end{equation}
Combining \eqref{10be} and \eqref{1+ber} we complete the proof of the Lemma.
\end{proof}

For any $z_i=(x_i,t_i),\, i=1,2$, denote the distance function $\bar{\r}$ on $\Si\times\R$ by
$$\bar{\r}_{z_1}(z_2)=\sqrt{\r_{x_1}^2(x_2)+(t_2-t_1)^2}.$$
For any $q,x\in\Si$ there are $\tilde{q}=(q,u(q))\in M$ and $\tilde{x}=(x,u(x))\in M$ such that $\bar{\r}_{\tilde{q}}(\tilde{x})=\sqrt{\r_q^2(x)+(u(x)-u(q))^2}$.
\begin{lemma}
Suppose the sectional curvature $R(x)\le K^2$ on $B_{i(q)}(q)$ with injective radius $i(q)$ at $q$, then we have
$$\De\bar{\r}_q^2(\tilde{x})\ge2+2(n-1)K\r_q(x)\cot(K\r_q(x))\qquad \mathrm{for}\ \ 0<\r_q(x)<\min\left\{i(q),\f{\pi}{2K}\right\}.$$
\end{lemma}
\begin{proof}
By the Hessian comparison theorem, for any $\xi\bot\f{\p}{\p\r_q}$ we have
$$\mathrm{Hess}_{\r_q}(\xi,\xi)\ge K\cot{(K\r_q)}|\xi|^2.$$
Let $\{e_i\}_{i=1}^n$ be a local orthonormal frame field of $M$. Note that $M$ is minimal, we obtain
\begin{equation}\aligned\label{Debr2}
\De\bar{\r}_q^2=&\sum_{i=1}^n\left(\na_{e_i}\na_{e_i}\bar{\r}_q^2-\left({\na_{e_i}e_i}\right)\bar{\r}_q^2\right)\\
=&\sum_{i=1}^n\left(\bn_{e_i}\bn_{e_i}\bar{\r}_q^2-\left({\bn_{e_i}e_i}\right)\bar{\r}_q^2\right)
+\sum_{i=1}^n\left({\bn_{e_i}e_i}-\na_{e_i}e_i\right)\bar{\r}_q^2\\
=&\De_{N}\bar{\r}_q^2-\overline{\mathrm{Hess}}_{\bar{\r}_q^2}(\nu,\nu)\\
=&\De_{\Si}\r_q^2+2-\f1{v^2}\mathrm{Hess}_{\r_q^2}(Du,Du)-\f{2}{v^2}.
\endaligned
\end{equation}
If $Du\neq0$, we set $(Du)^T=Du-\left\lan Du,\f{\p}{\p\r_q}\right\ran \f{\p}{\p\r_q}$. Let $\{E_\a\}_{\a=1}^{n-1}\bigcup\f{\p}{\p\r_q}$ be an orthonormal basis of $T\Si$ with $E_1=(Du)^T\left|(Du)^T\right|^{-1}$. Combining $\mathrm{Hess}_{\r_q^2}\left(E_\a,\f{\p}{\p\r_q}\right)=0$ and $\mathrm{Hess}_{\r_q^2}\left(\f{\p}{\p\r_q},\f{\p}{\p\r_q}\right)=2$ give
\begin{equation}\aligned
\mathrm{Hess}_{\r_q^2}(Du,Du)=\mathrm{Hess}_{\r_q^2}\left((Du)^T,(Du)^T\right)+2\left\lan Du,\f{\p}{\p\r_q}\right\ran^2.
\endaligned
\end{equation}
Hence
\begin{equation}\aligned
\De\bar{\r}_q^2=&\sum_\a\mathrm{Hess}_{\r_q^2}(E_\a,E_\a)+4-\f1{v^2}\mathrm{Hess}_{\r_q^2}\left((Du)^T,(Du)^T\right)-\f2{v^2}\left\lan Du,\f{\p}{\p\r_q}\right\ran^2-\f{2}{v^2}\\
=&\sum_\a\mathrm{Hess}_{\r_q^2}(E_\a,E_\a)+2-\f{\left|(Du)^T\right|^2}{v^2}\mathrm{Hess}_{\r_q^2}(E_1,E_1)+\f2{v^2}\left|(Du)^T\right|^2\\
\ge&2(n-2)K\r_q\cot(K\r_q)+2+\left(2-2\f{\left|(Du)^T\right|^2}{v^2}\right)K\r_q\cot(K\r_q)+\f2{v^2}\left|(Du)^T\right|^2\\
\ge&2(n-1)K\r_q\cot(K\r_q)+2.
\endaligned
\end{equation}
If $|Du|=0$, clearly $\De\bar{\r}_q^2=\De_{\Si}\r_q^2\ge2(n-1)K\r_q\cot(K\r_q)+2$.
\end{proof}



Suppose that $\Si$ satisfies conditions C1), C2) and C3). For sufficiently small $\de>0$ depending only on $n,c,V_{\Si}$ and any fixed $q\in\p B_{\Om r}(p)$ with $\Om=\left(\f{\sqrt{c}}{\de}+1\right)$, by \cite{CGT} the injectivity radius at $q$ satisfies $i(q)\ge r$ and
$$d(p,x)\ge\f{\sqrt{c}}{\de}\ r,\qquad \mathrm{for\ any}\ \ x\in B_r(q).$$
Then by condition C3)
\begin{equation}\aligned
|R(x)|\le\f{\de^2}{r^2}, \qquad \mathrm{for\ any}\ \ x\in B_r(q).
\endaligned
\end{equation}
Hence $\r_q(x)$ is smooth for $x\in B_r(q)\setminus\{q\}$.
For $\tilde{q}=(q,u(q))\in M$, we
denote $\mathbb{B}_s(\tilde{q})=\{z\in N|\ \bar{\r}_{\tilde{q}}(z)<s\}$ and $D_s(\tilde{q})=\mathbb{B}_s(\tilde{q})\cap M$. If $\tilde{x}=(x,u(x))\in D_s(\tilde{q})$, then obviously $x\in B_s(q)$.

For any $t\in[0,1)$ we have $\cos t\ge1-t$, then
$$\left(\tan t-\f t{1-t}\right)'=\f1{\cos^2 t}-\f1{(1-t)^2}\le0.$$
So on $[0,1)$
$$\tan t\le\f t{1-t}.$$
Hence on $D_r(\tilde{q})$ we have
$$\De\bar{\r}_{\tilde{q}}^2(\tilde{x})\ge2+2(n-1)\left(1-\f\de r\r_q(x)\right)\ge2n-\f{2n\de\r_q(x)}{r}.$$

For any smooth function $f$ on $M$, combining the above inequalities we get
\begin{equation}\aligned\label{2nfDsbq}
&2n\int_{D_s(\tilde{q})}f\left(1-\f{\de\r_q}{r}\right)
\le\int_{D_s(\tilde{q})}f\De\bar{\r}_{\tilde{q}}^2=\int_{D_s(\tilde{q})}\div\left(f\na\bar{\r}_{\tilde{q}}^2\right)-\int_{D_s(\tilde{q})}\na f\cdot\na\bar{\r}_{\tilde{q}}^2\\
=&\int_{\p D_s(\tilde{q})}f\na\bar{\r}_{\tilde{q}}^2\cdot\f{\na\bar{\r}_{\tilde{q}}}{|\na\bar{\r}_{\tilde{q}}|}-\int_{D_s(\tilde{q})}\div\left((\bar{\r}_{\tilde{q}}^2-s^2)\na f\right)
+\int_{D_s(\tilde{q})}(\bar{\r}_{\tilde{q}}^2-s^2)\De f\\
=&2s\int_{\p D_s(\tilde{q})}f|\na\bar{\r}_{\tilde{q}}|+\int_{D_s(\tilde{q})}(\bar{\r}_{\tilde{q}}^2-s^2)\De f.
\endaligned
\end{equation}
Since
\begin{equation}\aligned
|\na\bar{\r}_{\tilde{q}}|^2=&\f1{4\bar{\r}_{\tilde{q}}^2}g^{ij}\p_i(\r_q^2+(u-u(q))^2)\cdot\p_j(\r_q^2+(u-u(q))^2)\\
=&\f1{\bar{\r}_{\tilde{q}}^2}\left(\r_q^2\left(1-\f{|Du\cdot D\r_q|^2}{v^2}\right)+2\r_q(u-u(q))\f{Du\cdot D\r_q}{v^2}+(u-u(q))^2\f{|Du|^2}{v^2}\right)\\
\le&\f1{\bar{\r}_{\tilde{q}}^2}\left(\r_q^2+(u-u(q))^2\f{1}{v^2}+(u-u(q))^2\f{|Du|^2}{v^2}\right)\\
=&1,
\endaligned
\end{equation}
we have
\begin{equation}\aligned
&\f{\p}{\p s}\left(s^{-n}\int_{D_s(\tilde{q})}\log v\right)=-n s^{-n-1}\int_{D_s(\tilde{q})}\log v+s^{-n}\int_{\p D_s(\tilde{q})}\f{\log v}{|\na\bar{\r}_{\tilde{q}}|}\\
\ge&-n s^{-n-1}\int_{D_s(\tilde{q})}\log v+s^{-n}\int_{\p D_s(\tilde{q})}\log v|\na\bar{\r}_{\tilde{q}}|\\
\ge&-n s^{-n-1}\int_{D_s(\tilde{q})}\log v+s^{-n}\f ns\int_{D_s(\tilde{q})}\left(1-\f{\de\r_q}{r}\right)\log v\\
\ge&-\f{n\de}{r}s^{-n}\int_{D_s(\tilde{q})}\log v.
\endaligned
\end{equation}
Let $\omega_n$ be the volume of the standard $n$-dimensional unit ball in Euclidean space. For $0<s\le r$ integrating the above inequality implies
\begin{equation}\aligned\label{mean_q}
\log v(\tilde{q})\le\f{e^{\f{n\de s}{r}}}{\omega_ns^{n}}\int_{D_s(\tilde{q})}\log v.
\endaligned
\end{equation}

We say that a function $f$ has \emph{at most linear growth} on $\Si$ if
$$\limsup_{x\rightarrow\infty}\f{|f(x)|}{\r(x)}<\infty.$$
and say that  $f$ has \emph{linear growth} on $\Si$ if
$$0<\limsup_{x\rightarrow\infty}\f{|f(x)|}{\r(x)}<\infty.$$
Denote $f_+=\max\{f,0\}$ and $f_-=\min\{f,0\}$.
\begin{theorem}\label{GEu}
Let $u$ be a minimal graphic function on a  complete Riemannian manifold $\Si$ which satisfies conditions C1), C2) and C3). Then we have gradient estimates
\begin{equation}\aligned\label{GradEstu}
|Du(x)|\le C_1e^{C_2r^{-1}\left(u(x)-\sup_{y\in B_{(\Om+1)r}(p)}u(y)\right)}
\endaligned
\end{equation}
for any $r>0$ and $x\in\p B_{\Om r}(p)$, where $C_1,C_2$ are positive constants depending only $n$, and $\Om$ is a constant depending only on $n,c,V_\Si$.
Moreover, if $u_+$ (or $u_-$) has at most linear growth, then $|Du|$
is uniformly bounded on all of  $\Si$.
\end{theorem}
\begin{proof}
For any $p\in\Si$ fixed, let $\Om=\left(\f{\sqrt{c}}{\de}+1\right)$ as before. For any $r>0$, $x\in\p B_{\Om r}(p)$, combining \eqref{mean_q} and \eqref{DuK} with $\be=1$ and $u(x)-u$ replacing $u$, we have
\begin{equation}\aligned
\log|Du(x)|\le&\log v(\tilde{x})\le\f{e^{n\de}}{\omega_nr^{n}}\int_{D_r(\tilde{x})}\log v\\
\le&11e^{n\de}\left(3+\f{\sup_{y\in B_{r}(x)}(u(x)-u(y))}{r}\right)\f{Vol(B_{3r})}{\omega_nr^n}\\
\le&11e^{n\de}\left(3+\f{\sup_{y\in B_{(\Om+1)r}(p)}(u(x)-u(y))}{r}\right)\f{Vol(B_{3r})}{\omega_nr^n},
\endaligned
\end{equation}
where $\tilde{x}=(x,u(x))$. Take $0<\de<1$, then invoking \eqref{VolSi} we obtain
\begin{equation}\aligned
\log|Du(x)|\le11(3e)^{n}\left(3+\f{u(x)-\sup_{y\in B_{(\Om+1)r}(p)}u(y)}{r}\right).
\endaligned
\end{equation}
Thus \eqref{GradEstu} holds. Obviously, we can substitute $u$ in
\eqref{GradEstu} by $-u$. Hence if $u_+$ or $u_-$ has at most linear
growth, then letting $r\rightarrow\infty$ implies that $|Du|$ is
uniformly bounded on all of $\Si$.
\end{proof}


\begin{lemma}\label{infsupmean}
Let $\Si$ be a Riemannian manifold with conditions C1), C2) C3), and
$u$ be a smooth  solution to \eqref{u} on $\Si$ with  at most linear growth. Then for any nonnegative subharmonic(superharmonic) function $\varphi^+(\varphi^{-})$ on $M$, we have
$$\sup_{D_{\la r}(q)}\varphi^+\le \f {C_1}{r^n}\int_{D_r(q)}\varphi^+\qquad and \qquad \inf_{D_{\la r}(q)}\varphi^-\ge \f {C_2}{r^n}\int_{D_r(q)}\varphi^-$$
for arbitrary $r>0$ and some constant $0<\la<1$ independent of $r$. Here $C_1,C_2>0$ are constants depending only on $n,c,V_{\Si},\la$.
\end{lemma}
\begin{proof}
In the manifold $\Si$  we have the Sobolev  and
Neumann-Poincar$\acute{\mathrm{e}}$ inequalities. So we have those
inequalities also on the manifold $M$ by the boundedness of $|Du|$. Then we can use De Giorgi-Moser-Nash's theory and the volume doubling condition to obtain the mean value inequality(see \cite{M} for the details).
\end{proof}
Denote $|Du|_0=\sup_{x\in\Si}|Du(x)|$.
Now we want to use the above Lemma to deduce the mean value equalities
for the bounded gradient of $u$.
\begin{lemma}
Let $\Si$ be a Riemannian manifold with conditions C1), C2) C3), and
$u$ be a smooth  solution to \eqref{u} on $\Si$ with  at most linear growth. Then we have mean value equalities on both exterior balls and interior balls:
\begin{equation}\aligned\label{supDu1}
|Du|_0^2=\lim_{r\rightarrow \infty}\f1{Vol(D_r(z))}\int_{D_r(z)}|Du|^2d\mu,
\endaligned
\end{equation}
and
\begin{equation}\aligned\label{supDu2}
|Du|_0^2=\lim_{r\rightarrow \infty}\f1{Vol(B_r(x))}\int_{B_r(x)}|Du|^2d\mu_{\Si}.
\endaligned
\end{equation}
\end{lemma}
\begin{proof}
Denote $\phi_{\max}\triangleq\sup_{z\in M}\log v(z)$.
If $u$ has  linear growth at most, for any $\tilde{q}\in M$ we have
\begin{equation}\aligned
\f {1}{r^n}\int_{D_r(\tilde{q})}\left(\phi_{\max}-\log v\right)\le C\left(\phi_{\max}-\log v(\tilde{q})\right),
\endaligned
\end{equation}
which implies for any $z\in M$
\begin{equation}\aligned\label{meanphimax}
\phi_{\max}=\lim_{r\rightarrow \infty}\f1{Vol(D_r(z))}\int_{D_r(z)}\log v<\infty.
\endaligned
\end{equation}
Since $e^{2\log v}-1=|Du|^2$ is a bounded subharmonic function on $M$,  we  obtain \eqref{supDu1}.

For any $0<\ep<|Du|_0$ and any fixed point $z=(x,u(x))$ there is an $r_0>0$ such that for any $r\ge r_0$
$$\f1{Vol(D_r(z))}\int_{D_r(z)}|Du|^2d\mu>|Du|_0^2-\ep^2.$$
Denote $\Om_r\triangleq\{(q,u(q))\in D_r(z)|\ |Du|^2(q)<|Du|_0^2-\ep\}$, then
\begin{equation*}\aligned
|Du|_0^2-\ep^2<&\f1{Vol(D_r(z))}\left(\int_{D_r(z)\setminus\Om_r}|Du|_0^2d\mu+\int_{\Om_r}|Du|^2d\mu\right)\\
\le&\left(1-\f{Vol(\Om_r)}{Vol(D_r(z))}\right)|Du|_0^2+\f{Vol(\Om_r)}{Vol(D_r(z))}\left(|Du|_0^2-\ep\right)\\
=&|Du|_0^2-\ep\f{Vol(\Om_r)}{Vol(D_r(z))},
\endaligned
\end{equation*}
which implies
$$Vol(\Om_r)\le\ep Vol(D_r(z)).$$
Let $\Om^T_r$ be the projection from $\Om_r$ to $\Si$ defined by
$$\{q\in\Si|\ (q,u(q))\in\Om_r\}=\{q\in\Si|\ (q,u(q))\in D_r(z),\ |Du|^2(q)<|Du|_0^2-\ep\}.$$
There exists a constant $C>1$ such that the projection of $D_{Cr}(z)$ contains $B_r(x)$ for any $r\ge r_0$ . By the definition of $D_{Cr}(z)$, it follows that
\begin{equation}\aligned
&\f1{Vol(B_r(x))}\int_{B_r(x)}|Du|^2d\mu_{\Si}\ge\f1{Vol(B_r(x))}\int_{B_r(x)\setminus\Om_{Cr}^T}\left(\sup_{\Si}|Du|^2-\ep\right) d\mu_{\Si}\\
\ge&\f1{Vol(B_r(x))}\int_{B_r(x)}\left(\sup_{\Si}|Du|^2-\ep\right) d\mu_{\Si}-\f1{Vol(B_r(x))}\int_{\Om_{Cr}^T}\left(\sup_{\Si}|Du|^2-\ep\right) d\mu_{\Si}\\
\ge&\sup_{\Si}|Du|^2-\ep-\f{Vol(\Om_{Cr})}{Vol(B_r(x))}\left(\sup_{\Si}|Du|^2-\ep\right)\\
\ge&\sup_{\Si}|Du|^2-\ep-\ep\f{Vol(D_{Cr}(z))}{Vol(B_r(x))}\left(\sup_{\Si}|Du|^2-\ep\right).
\endaligned
\end{equation}
Let $r\rightarrow\infty$, then $\ep\rightarrow0$ implies that \eqref{supDu2} holds.
\end{proof}

\begin{theorem}\label{Liusub}
If the minimal graphic function $u$ on a complete manifold with conditions C1), C2) and C3) has sub-linear growth for its negative part, namely,
$$\limsup_{x\rightarrow\infty}\f{|u_-(x)|}{\r(x)}=0,$$
where $u_-=\min\{u,0\}$. Then $u$ is a constant.
\end{theorem}
\begin{proof}
From Theorem \ref{GEu}, $|Du|$ is globally bounded. For any small $\de>0$, there is a $C_\de>0$ such that $u(x)\ge-C_\de-\de\r(x)$. Lemma \ref{infsupmean} implies a Harnack inequality for positive harmonic functions on $M$. Hence there is an absolute constant $C>1$ so that for any $r>0$ and $q\in M$ we have
$$\sup_{x\in D_{\la r}(q)}\left(u(x)-\inf_{D_r(q)}u+1\right)\le C\inf_{x\in D_{\la r}(q)}\left(u(x)-\inf_{D_r(q)}u+1\right)\le C\left(u(q)-\inf_{D_r(q)}u+1\right).$$
Thus
\begin{equation}\aligned
\sup_{x\in D_{\la r}(q)}u(x)\le Cu(q)+(C-1)\left(1-\inf_{D_r(q)}u\right)\le Cu(q)+(C-1)\left(1+C_\de+\de r\right)
\endaligned
\end{equation}
Therefore, for any $\ep>0$ there is $r_0>0$ such that for any $x$ with $\r(x)\ge r_0$ we have
$$|u(x)|<\ep\r(x).$$
For any $r\ge r_0$, let $\eta$ be a Lipschitz function on $M$ with $\mathrm{supp}\eta\subset D_{2r}$ with $\eta\big|_{D_r}\equiv1$ and $|\na\eta|\le\f 1r$ on $D_{2r}\setminus D_r$. Due to $\De u=0$, we see that
\begin{equation}\aligned
0=-\int_Mu\eta^2\De u=\int_M\na u\cdot\na(u\eta^2)=\int_M|\na u|^2\eta^2+2\int_M u\eta\na u\cdot\na\eta.
\endaligned
\end{equation}
From the Cauchy inequality we conclude that
\begin{equation}\aligned
\int_{D_r}|\na u|^2\le\int_M|\na u|^2\eta^2\le4\int_M |\na\eta|^2u^2\le\f4{r^2}\int_{D_{2r}\setminus D_r}u^2\le16\ep^2Vol(D_{2r}).
\endaligned
\end{equation}
With $|\na u|^2=\f{|Du|^2}{v^2}$, we get
\begin{equation}\aligned
\sup_{y\in \Si}|Du|^2(y)=&\lim_{r\rightarrow \infty}\f1{Vol(D_r)}\int_{D_r}|Du|^2d\mu\\
\le& \limsup_{r\rightarrow \infty}\f{1+|Du|_0^2}{Vol(D_r)}\int_{D_{r}}|\na u|^2d\mu\\
\le&16\ep^2(1+|Du|_0^2)\limsup_{r\rightarrow\infty}\f{Vol(D_{2r})}{Vol(D_r)}.
\endaligned
\end{equation}
Forcing $\ep\rightarrow0$ gives $|Du|\equiv0$, namely, $u$ is a constant.
\end{proof}
Actually, if $\Si$ is not Euclidean space, we can give a stronger theorem by replacing sub-linear growth by linear growth in the following section.

\section{A Liouville theorem via splitting for tangent cones at infinity}

Let $\Si$ be an $n$-dimensional manifold with conditions C1), C2), C3), and $u$ be a smooth linear growth solution to \eqref{u}.
We shall now derive pointwise estimates for the Hessian of $u$. We
first rewrite \eqref{u}, in order to apply elliptic regularity theory
as can found, for instance in \cite{GT,J}.  In a local coordinate, \eqref{u} is
\begin{equation}
\p_j\left(\sqrt{\si}\f{\si^{ij}u_i}{\sqrt{1+|Du|^2}}\right)=0,
\end{equation}
where $\sqrt{\si}=\sqrt{\det\si_{kl}}$. Take derivatives and set $w=\p_ku$ to turn this equation into
\begin{equation}
\p_j\left(\f{\sqrt{\si}}{\sqrt{1+|Du|^2}}\left(\si^{ij}-\f{u^iu^j}{1+|Du|^2}\right)w_i\right)+\p_jf^j_k=0,
\end{equation}
where
$$f^j_k=\f{u_i}{\sqrt{1+|Du|^2}}\p_k\left(\sqrt{\si}\si^{ij}\right)-\f12\sqrt{\si}\si^{ij}\f{u_iu_pu_q}{\left(1+|Du|^2\right)^{3/2}}\p_k\si^{pq}.$$

Define an operator $L$ on $C^2(\Si)$ by
$$Lf=\p_j\left(\f{\sqrt{\si}}{\sqrt{1+|Du|^2}}\left(\si^{ij}-\f{u^iu^j}{1+|Du|^2}\right)f_i\right),$$
then
$$Lw+\p_jf^j_k=0.$$
Since $\Si$ satisfies conditions C1), C2) and C3), by \cite{AC,GW,Ps} there
exist harmonic coordinates satisfying the estimates of \cite{JK}. That is, for a fixed point $p\in\Si$, there are positive constants $\a'=\a'(n,c,V_{\Si}),\th=\th(n,c,V_{\Si})\in(0,1)$ and $C=C(n,c,V_{\Si},\a')$ such that for any $q\in \p B_r(p)$ and $r>0$ there is a harmonic coordinate system $\{x_i:\ i=1,\cdots,n\}$ on $B_{\th r}(q)$ satisfying
\begin{equation}\aligned\label{harmcoor}
\De_\Si\, x_i\equiv0,\qquad (\si_{ij})_{n\times n}\ge\f1C,\qquad \si_{ij}+r|D\si_{ij}|+r^{1+\a'}[D\si_{ij}]_{\a',B_{\th r}(q)}\le C,
\endaligned
\end{equation}
where $0<\a'<1$, $\si_{ij}=\si\left(\f \p{\p x_i},\f \p{\p x_j}\right)$ and
$$[\varphi]_{\a',B_{\th r}(q)}=\sup_{x,y\in B_{\th r}(q),x\neq y}\f{|\varphi(x)-\varphi(y)|}{d(x,y)^{\a'}}.$$

\begin{lemma}
For any $q\in\p B_{r}(p)\subset\Si$, $r>0$ and $s\le \th r$ we have
\begin{equation}\aligned
\mathrm{osc}_{B_s(q)}Du\le Cs^\a r^{-\a},
\endaligned
\end{equation}
where $C=C(n,|Du|_0,c,V_{\Si})$ and $\a=\a(n,|Du|_0,c,V_{\Si})\le\a'<1$ are positive constants.
\end{lemma}
\begin{proof}
For any fixed $s\le\f14\th r$, denote $M_4(s)=\sup_{B_{4s}(q)}w$, $m_4(s)=\inf_{B_{4s}(q)}w$, $M_1(s)=\sup_{B_{s}(q)}w$, $m_1(s)=\inf_{B_{s}(q)}w$. Then we have
$$L(M_4-w)=\p_if^i_k,\qquad L(w-m_4)=-\p_if^i\qquad \mathrm{on}\ B_{\th r}(q).$$
Due to \eqref{harmcoor}, it is not hard to find out that $|f^i_k|\le\f
Cr$ on $B_{\th r}(q)$. By the weak Harnack inequality (Theorem 8.18 of \cite{GT}), we have
\begin{equation}\aligned\label{Mw4}
s^{-n}\int_{B_{2s}(q)}\left(M_4(s)-w\right)d\mu_{\Si}\le C\left(M_4(s)-M_1(s)+\f sr\right),
\endaligned
\end{equation}
and
\begin{equation}\aligned\label{wm4}
s^{-n}\int_{B_{2s}(q)}\left(w-m_4(s)\right)d\mu_{\Si}\le C\left(m_1(s)-m_4(s)+\f sr\right),
\endaligned
\end{equation}
where $C=C(n,|Du|_0,c,V_{\Si})$.
Denote $\omega(s)=\mathrm{osc}_{B_s(q)}w=M_1(s)-m_1(s)$. Combining \eqref{VolSi}, \eqref{Mw4} and \eqref{wm4} gives
\begin{equation}\aligned
2^nV_{\Si}\ \omega(4s)\le\f{Vol(B_{2s}(q))}{s^n}\omega(4s)\le C\left(\omega(4s)-\omega(s)+\f {2s}r\right),
\endaligned
\end{equation}
which implies that there is a $\g\in(0,1)$ such that for all $s\in[0,\f14\th r]$
$$\omega(s)\le \g\omega(4s)+\f {2s}r.$$
By an iterative trick(see Lemma 8.23 in \cite{GT}), we complete the proof.
\end{proof}
The above Lemma implies the following H$\mathrm{\ddot{o}}$lder
continuity for the gradient of $u$.
\begin{corollary}
For any $q\in\p B_{r}(p)\subset\Si$, $r>0$ we have
\begin{equation}\aligned
{[Du]}_{\a,B_{\th r}(q)}\le C r^{-\a},
\endaligned
\end{equation}
where $C=C(n,|Du|_0,c,V_{\Si})$ and $\a=\a(n,|Du|_0,c,V_{\Si})<1$ are positive constants.
\end{corollary}
Standard elliptic regularity theory (the scale-invariant Schauder
estimates, see \cite{GT,J}) implies that there exists a constant $C=C(n,|Du|_0,c,V_{\Si},\a)$ such that for $q\in\p B_r(p)$
\begin{equation}\aligned\label{D2u}
\sup_{B_{\f{\th r}2}(q)}|D^2u|\le C r^{-2}\sup_{B_{\th r}(q)}|u|,
\endaligned
\end{equation}
and
\begin{equation}\aligned\label{D2au}
{[D^2u]}_{\a,B_{\f{\th r}2}(q)}\le C r^{-2-\a}\sup_{B_{\th r}(q)}|u|.
\endaligned
\end{equation}

\begin{theorem}
If $u$ is solution to \eqref{u} with linear growth,
then
\begin{equation}\aligned\label{rD2u0}
\limsup_{r\rightarrow\infty}\left(r\sup_{\p B_{r}(p)}|D^2u|\right)=0.
\endaligned
\end{equation}
\end{theorem}
\begin{proof}
For the fixed point $p\in\Si$, there is a constant $C$ such that $\{(x,u(x))|\ x\in B_r(p)\}\subset D_{Cr}(p)$.
Let $b$ be the function defined on $\Si$ in section 2.
Together with \eqref{Deb} and the computation in \eqref{Debr2}, we have
\begin{equation}\aligned
\De b=&\De_{\Si}b-\f1{v^2}\mathrm{Hess}_{b}(Du,Du)\\
=&(n-1)\f{|Db|^2}b-\f1{2v^2b}\left(\mathrm{Hess}_{b^2}(Du,Du)-2\lan Db,Du\ran^2\right).
\endaligned
\end{equation}
By the properties of the function $b$, there is a large $r_0$ so that for any $x$ with $b(x)\ge r_0$ one has
$$|\De b|\le\f {n+2}b.$$

Denote $U_s=\{(x,u(x))\in M|\ x\in\{b<s\}\}$ for $s>0$. Let $\z$ be a nonnegative smooth function on $\R^+$ with $\mathrm{supp}\z\subset [0,2r]$, $\z\big|_{[0,r]}\equiv1$, $|\z'|\le\f Cr$ and $|\z''|\le\f C{r^2}$. Set
$$\eta(x)=\z(b(x))\qquad \mathrm{for\ any}\ x\in\Si.$$
Then $\eta$ is a smooth function with $\mathrm{supp}\eta\subset U_{2r}$, $\eta\big|_{U_{r}}\equiv1$, $|\na\eta|\le\f Cr$ and $|\De\eta|\le\f C{r^2}$ for sufficiently large $r$. Recalling \eqref{Logv} and $b(x)\le\r(x)$ we obtain
\begin{equation}\aligned\label{HessuB2}
\int_{B_r}|\mathrm{Hess}_u|^2d\mu_{\Si}\le&\int_{\{b<r\}}|\mathrm{Hess}_u|^2 d\mu_{\Si}\le \int_{U_{2r}}|\mathrm{Hess}_u|^2\eta d\mu\\
\le& C\int_{U_{2r}}|B|^2\eta\le C\int_{U_{2r}}\eta\De\left(\log v-\phi_{\max}\right)\\
=&C\int_{U_{2r}}\left(\log v-\phi_{\max}\right)\De\eta\le\f C{r^2}\int_{U_{2r}}\left(\phi_{\max}-\log v\right),
\endaligned
\end{equation}
where $\phi_{\max}=\sup_{z\in M}\log v(z)$ as before. Due to \eqref{meanphimax} we see that
\begin{equation}\aligned\label{BrHessu}
\lim_{r\rightarrow\infty}\left(\f1{r^{n-2}}\int_{B_r}|\mathrm{Hess}_u|^2d\mu_\Si\right)=0.
\endaligned
\end{equation}

From \eqref{D2u} we have
\begin{equation}\aligned
\sup_{\p B_{r}(p)}|D^2u|\le\f C r
\endaligned
\end{equation}
for some fixed $C=C(n,|Du|_0,c,V_{\Si},\a)$. If
\begin{equation}\aligned
\limsup_{r\rightarrow\infty}\left(r\sup_{\p B_{r}(p)}|D^2u|\right)>0,
\endaligned
\end{equation}
then there exist $\ep>0$, $r_i\rightarrow\infty$ and $q_i\in\p B_{r_i}(p)$ such that
\begin{equation}\aligned\label{riD2uqiep}
r_i|D^2u(q_i)|\ge\ep.
\endaligned
\end{equation}
By \eqref{D2au}, we conclude that
\begin{equation}\aligned
{[D^2u]}_{\a,B_{\f{\th r_i}2}(q_i)}\le C r_i^{-1-\a}.
\endaligned
\end{equation}
There is a sufficiently small $\de=\de(n,|Du|_0,c,V_{\Si},\a,\ep)\in(0,\f{\th}2)$ such that for any $y\in B_{\de r_i}(q_i)$
\begin{equation}
|D^2u(q_i)|-|D^2u(y)|\le{[D^2u]}_{\a,B_{\f{\th r_i}2}(q_i)}d(y,q_i)^\a\le C r_i^{-1-\a}\left(\de r_i\right)^\a<\f{\ep}{2r_i},
\end{equation}
which together with (\ref{riD2uqiep}) implies
\begin{equation}\aligned
|D^2u(y)|\ge\f{\ep}{2r_i}.
\endaligned
\end{equation}
Hence
\begin{equation}\aligned
&\f1{r_i^{n-2}}\int_{B_{(1+\de)r_i}(p)\setminus B_{(1-\de)r_i}(p)}|D^2u|^2d\mu_{\Si}\ge\f1{r_i^{n-2}}\int_{B_{\de r_i}(q_i)}|D^2u|^2d\mu_{\Si}\\
\ge&\f1{r_i^{n-2}}\int_{B_{\de r_i}(q_i)}\f{\ep^2}{4r_i^2}d\mu_{\Si}\ge\f{\ep^2}4\de^n V_{\Si}.\\
\endaligned
\end{equation}
Letting $r_i\rightarrow\infty$ deduces a contradiction to \eqref{BrHessu}. This completes the proof.
\end{proof}

Re-scale the metric by $\si\rightarrow r^{-2}\si$ and denote the inner
product, norm, gradient, Hessian and volume element for this re-scaled
metric by $\lan\cdot,\cdot\ran_r$, $|\cdot |_r$, $D^r$,
$\mathrm{Hess}^r$ and $d\mu_r$, respectively. Set
$\tilde{u}_r=r^{-1}u$ and let $B^r_s$ be the ball with radius $s$ and centered at $p$ for this re-scaled metric. Namely, $B^r_s$ is the ball with radius $s$ and centered at $p$ in $r^{-2}\Si=(\Si,r^{-2}\si,p)$. Then for any fixed $r_0>0$ \eqref{supDu2} implies
\begin{equation}\aligned\label{f1}
\lim_{r\rightarrow\infty}\f1{Vol(B_{r_0}^r)}\int_{B_{r_0}^r}|D^r\tilde{u}_r|_r^2d\mu_r=\sup_{\Si}|Du|^2,
\endaligned
\end{equation}
and \eqref{rD2u0} implies that there is a constant $C>0$ so that
\begin{equation}\aligned\label{f2}
\lim_{r\rightarrow\infty}\sup_{\p B_{r_0}^r}|\mathrm{Hess}^r_{\tilde{u}_r}|_r=0\qquad \mathrm{and}\qquad \sup_{\p B_{r_0}^s}|\mathrm{Hess}^s_{\tilde{u}_r}|_s<\f C{r_0}\quad \mathrm{for\ any}\ s>0.
\endaligned
\end{equation}

For any $x,y\in B^r_{r_0}(p)$ there exists a minimal normal geodesic
$\g_{xy}$ connecting $x$ and $y$ such that $\g_{xy}(0)=x$, $\g_{xy}(1)=y$ and
$|\dot{\g}_{xy}|=l_{xy}$ is a constant. Clearly, $l_{xy}\le2r_0$. Parallel
translating the vector $D^r\tilde{u}_r(x)$ along $\g_{xy}(t)$ produces a
unique vector at $y$ denoted by $D^r_{\pi_y}\tilde{u}_r(x)$. Denote by $C$  a  constant depending only on $n$. We have
\begin{equation}\aligned
\left|D^r_{\pi_y}\tilde{u}_r(x)-D^r\tilde{u}_r(y)\right|_r=&\left|D^r\tilde{u}_r(x)-D^r_{\pi_x}\tilde{u}_r(y)\right|_r\\
\le&C\cdot l_{xy}\int_0^{1}|\mathrm{Hess}_{\tilde{u}_r}|_r(\g_{xy}(t))dt.
\endaligned
\end{equation}
For any fixed $\ep>0$, together with \eqref{f1}\eqref{f2} and the above inequality we get
\begin{equation}\aligned\label{Drpixy}
\lim_{r\rightarrow\infty}\sup_{\g_{xy}\in B^r_{r_0}\setminus B^r_{\ep}}\left|D^r_{\pi_y}\tilde{u}_r(x)-D^r\tilde{u}_r(y)\right|_r=0,
\endaligned
\end{equation}
and
\begin{equation}\aligned\label{Drux0}
\lim_{r\rightarrow\infty}\inf_{x\in B^r_{r_0}\setminus B^r_{\ep}}|D^r\tilde{u}_r(x)|_r=\sup_{\Si}|Du|.
\endaligned
\end{equation}

Note that $|Du|_0\triangleq\sup_{\Si}|Du|>0$. Suppose $u(p)=\tilde{u}_r(p)=0$. Set
$$\G_r=\tilde{u}_r^{-1}(0)=\{x\in r^{-2}\Si|\ \tilde{u}_r(x)=0\}.$$
For any $x\in B_{r_0}^r$, there exists $x_0\in B_{2r_0}^r\cap\G_r$ so that
$$d_r(x,x_0)=d_r(x,\G_r)\triangleq\inf_{z\in\G_r}d_r(x,z).$$
Here $d_r$ is the distance function on $r^{-2}\Si$. Set the signed distance function $d_{\G_r}(x)=d_r(x,\G_r)$ for $\tilde{u}_r(x)\ge0$, and $d_{\G_r}(x)=-d_r(x,\G_r)$ for $\tilde{u}_r(x)\le0$.


\begin{lemma}\label{ury}
For any fixed $r_0>0$ there is
\begin{equation}
\lim_{r\rightarrow\infty}\sup_{y\in B_{r_0}^r} \left|\tilde{u}_r(y)-|Du|_0\cdot d_{\G_r}(y)\right|=0
\end{equation}
\end{lemma}
\begin{proof}
For any $r>0$ there is a $y_r\in B^r_{r_0}$ such that
$$\big|\tilde{u}_r(y_r)-|Du|_0\cdot d_{\G_r}(y_r)\big|=\sup_{z\in B_{r_0}^r} \big|\tilde{u}_r(z)-|Du|_0\cdot d_{\G_r}(z)\big|.$$
Without loss of generality, we suppose $\tilde{u}_r(y_r),d_{\G_r}(y_r)>0$. Clearly,
$$\tilde{u}_r(y_r)-|Du|_0\cdot d_{\G_r}(y_r)\le0.$$

For any $\ep>0$ we set
$$\G_r^\ep=\tilde{u}_r^{-1}(\ep|Du|_0)=\{x\in r^{-2}\Si|\ \tilde{u}_r(x)=\ep|Du|_0\}.$$
Hence, for any $x_r\in\G_r^\ep$ one has $x_r\in r^{-2}\Si\setminus B^r_\ep$. For any fixed $\ep>0$ from \eqref{f1} and \eqref{f2} there is an $r_*>0$ so that for any $r\ge r_*$ we have $|D^r\tilde{u}_r|_r\Big|_{\G_r^\ep\cap B^r_{2r_0}}\ge\f12|Du|_0$.
Clearly there exist $z_r\in\G_r^\ep$ satisfying $d_r(y_r,z_r)=d_{\G_r^\ep}(y_r)$ and a unique normal geodesic  $\g_r$ connecting $y_r,z_r$ with $\g_r(0)=z_r$, $\g_r(l_r)=y_r$ and $|\dot{\g}_r|_r=1$, where $l_r=d_{\G_r^\ep}(y_r)$. Smoothness of $\G_r^\ep$ implies $\dot{\g}_r(0)=-D^r\tilde{u}_r(z_r)/|D^r\tilde{u}_r(z_r)|_r$,
then
\begin{equation}\aligned
\tilde{u}_r(y_r)=&\tilde{u}_r(y_r)-\tilde{u}_r(z_r)=\int_0^{l_r}\left\lan D^r\tilde{u}_r(\g_r(t)),\dot{\g}_r(t)\right\ran_rdt\\
=&\int_0^{l_r}\left\lan D^r\tilde{u}_r(\g_r(t))-D^r_{\pi_{\g}(t)}\tilde{u}_r(z_r),\dot{\g}_r(t)\right\ran_rdt+\int_0^{l_r}|D^r\tilde{u}_r(z_r)|_rdt\\
\ge&-\int_0^{l_r}\left|D^r_{\pi_{\g_r}(t)}\tilde{u}_r(z_r)-D^r\tilde{u}_r(\g_r(t))\right|_rdt+|D^r\tilde{u}_r(z_r)|_r\ d_{\G_r^\ep}(y_r)\\
\ge&-C\int_0^{l_r}\int_0^t\left|\mathrm{Hess}_{\tilde{u}_r}(\g_r(s))\right|_rdsdt+|D^r\tilde{u}_r(z_r)|_r\ d_{\G_r^\ep}(y_r).
\endaligned
\end{equation}
Combining \eqref{f2} and \eqref{Drux0} implies
\begin{equation}\aligned
\liminf_{r\rightarrow\infty}\left(\tilde{u}_r(y_r)-|Du|_0\cdot d_{\G_r^\ep}(y_r)\right)\ge0.
\endaligned
\end{equation}
Letting $\ep\rightarrow0$ completes the proof.
\end{proof}
\begin{remark}\label{urGsr}
Analogously to the proof of Lemma \ref{ury}, we have
$$\lim_{r\rightarrow\infty}\sup_{y\in B_{r_0}^r} \left|\tilde{u}_r(y)-|Du|_0 \left(d_{\G_r^s}(y)+s\right)\right|=0$$
for $s\in\R$, where $\G_r^s=\tilde{u}_r^{-1}(s|Du|_0)=\{x\in r^{-2}\Si|\ \tilde{u}_r(x)=s|Du|_0\}.$
\end{remark}

For any $x,y\in\G_r^s$, let $\g_{r,xy}^s$ be a normal geodesic joining $x$ to $y$ with length $l_{r,xy}^s$. Since
$$\f{\p^2}{\p t^2}\tilde{u}_r\left(\g_{r,xy}^s(t)\right)=\mathrm{Hess}_{\tilde{u}_r}\left(\dot{\g}_{r,xy}^s(t),\dot{\g}_{r,xy}^s(t)\right),$$
then by the Newton-Leibniz formula we conclude that
\begin{equation}\aligned\label{rxyr0urg}
\lim_{r\rightarrow\infty}\sup_{x,y\in\G^s_r\cap B^r_{r_0}}\left(\sup_{t\in[0,l_{r,xy}^s]}\left|\tilde{u}_r\left(\g_{r,xy}^s(t)\right)-s|Du|_0\right|\right)=0.
\endaligned
\end{equation}
For any sequence $r_i\rightarrow\infty$ there is a subsequence $r_{i_j}$ such that $r_{i_j}^{-1}\Si=(\Si,r_{i_j}^{-1}\si,p)$ converges to a regular metric cone $\Si_0$ with vertex $o$ in the pointed Gromov-Hausdorff sense. Clearly, the geodesic $\g_{r_{i_j},px}^s$ should converge to a radial line starting from $o$ in $\Si_0$. Therefore, combining \eqref{rxyr0urg} and Remark \ref{urGsr} we know that $\G_{r_{i_j}}^s$ must converge to an $(n-1)$-dimensional cone $C\mathcal{Y}_s$ in $\Si_0$. Moreover, for any $z_1,z_2\in C\mathcal{Y}_s$ the geodesic joining $z_1$ and $z_2$ in $\Si_0$ must live in $C\mathcal{Y}_s$.
Let $\Omega$ be a connected component of $\Si_0\setminus C\mathcal{Y}_s$ and $z_3,z_4\in\Omega$. Then the geodesic $\g_{z_3,z_4}$ joining $z_3$ and $z_4$ in $\Si_0$ satisfies  $\g_{z_3,z_4}\setminus C\mathcal{Y}_s\subset\Om$.

Consider $y,z\in\{x\in r^{-2}\Si|\ \tilde{u}_r(x)\ge0\}$.
For any $x\in\G_r$, $d_{\G_r}(y)\le d_r(y,x)\le d_r(x,z)+d_r(y,z)$. Taking all the points in $\G_r$ it follows that
$$d_{\G_r}(y)\le d_{\G_r}(z)+d_r(y,z).$$
The same method implies
$$d_{\G_r}(z)\le d_{\G_r}(y)+d_r(y,z).$$
Hence
\begin{equation}\aligned\label{dGrzy}
|d_{\G_r}(z)-d_{\G_r}(y)|\le d_r(y,z).
\endaligned
\end{equation}

For some fixed point $z$ with $\tilde{u}_r(z)>0$ and small $\ep>0$, choose $\de>0$ sufficiently small, then there is an $x\in\G_r$ such that
$$d_{\G_r}(z)+\ep\de\ge d_r(z,x).$$
Let $l(t)$ be a minimal geodesic  connecting $z$ and $x$, and denote $y\in l(t)\cap\p B^r_\de$. Then
$$d_{\G_r}(z)+\ep\ d_r(z,y)\ge d_r(y,x)+d_r(z,y)\ge d_{\G_r}(y)+d_r(z,y).$$
Combining this with \eqref{dGrzy} we conclude that $|D^rd_{\G_r}|_r\equiv1$ almost everywhere.
\begin{lemma}\label{Drur}
For any fixed $r_1>0$
\begin{equation}\aligned
\lim_{r\rightarrow\infty}\int_{B^r_{r_1}}\big|D^r\left(\tilde{u}_r-|Du|_0\cdot d_{\G_r}\right)\big|_r^2=0.
\endaligned
\end{equation}
\end{lemma}
\begin{proof}
Let $\eta$ be a Lipschitz function with $\eta\big|_{B^r_{r_1}}\equiv0$, $|D^r\eta|_r\le1$ and $\mathrm{supp}\ \eta\subset \overline{B}^r_{2r_1}$. Let $\lan\cdot,\cdot\ran_r$ and $\mathrm{div}_r$ be the inner product and divergence in $r^{-2}\Si$, then
\begin{equation}\aligned
&\int_{B^r_{r_1}}\left|D^r\left(\tilde{u}_r-|Du|_0\cdot d_{\G_r}\right)\right|_r^2
\le\int_{B^r_{2r_1}}\left|D^r\left(\tilde{u}_r-|Du|_0\cdot d_{\G_r}\right)\right|_r^2\eta^2\\
=&\int_{B^r_{2r_1}}\left(|Du|_0^2\cdot|D^rd_{\G_r}|_r^2-|D^r\tilde{u}_r|_r^2\right)\eta^2-2\int_{B^r_{2r_1}}\left\lan D^r\tilde{u}_r,D^r\left(|Du|_0\cdot d_{\G_r}-\tilde{u}_r\right)\right\ran_r\eta^2\\
=&\int_{B^r_{2r_1}}\left(|Du|_0^2-|D^r\tilde{u}_r|_r^2\right)\eta^2\\
&+2\int_{B^r_{2r_1}} \left(2\eta\lan D^r\eta,D^r\tilde{u}_r\ran_r+\eta^2\div_r\left( D^r\tilde{u}_r\right)\right)\left(|Du|_0\cdot d_{\G_r}-\tilde{u}_r\right).
\endaligned
\end{equation}
Together with \eqref{f1}\eqref{f2} and Lemma \eqref{ury}, we complete the proof.
\end{proof}

If $(Z,d)$ is a metric space and $S_1,S_2\subset Z$, then we set
$$d(S_1,S_2)=\inf\{d(s_1,s_2)|\ s_1\in S_1,\ s_2\in S_2\},\quad B(S_1,\ep)=\{z\in Z|\ d(z,S_1)<\ep\}.$$
And we define \emph{Hausdorff distance} $d_H$ on $S_1,S_2$ by
$$d_H(S_1,S_2)=\inf\{\ep>0|\ S_1\subset B(S_2,\ep),\ S_2\subset B(S_1,\ep)\}.$$
If $Z_1,Z_2$ are both metric spaces, then an \emph{admissible metric} on the disjoint union $Z_1\coprod Z_2$ is a metric that extends the given metrics on $Z_1$ and $Z_2$. With this one can define the \emph{Gromov-Hausdorff distance} as
$$d_{GH}(Z_1,Z_2)=\inf\{d_H(Z_1,Z_2)|\ \mathrm{adimissible\ metrics\ on}\ Z_1\coprod Z_2\}.$$

Now we consider a mapping $B^r_{r_1}\rightarrow B_{(p,0)}({r_1},\G_r\times\R):\ y\mapsto(y_0,|Du|_0^{-1}\tilde{u}_r(y))$, where $y_0\in\G_r$ satisfies $d_r(y,y_0)=d_{\G_r}(y)$ and $B_{(p,0)}({r_1},\G_r\times\R)$ denotes the ball in $\G_r\times\R$ with radius $r_1$ and centered at $(p,0)$. Together with Lemma \ref{ury} and Lemma \ref{Drur}, using Theorem 3.6 in \cite{CC}, we obtain
\begin{equation}\aligned\label{GHSigma}
\lim_{r\rightarrow\infty}d_{GH}\big(B_{r_1}^r,B_{(p,0)}({r_1},\G_r\times\R)\big)=0.
\endaligned
\end{equation}
In fact, we can also obtain \eqref{GHSigma} through the following Lemma.
\begin{lemma}
For any fixed $r_1>0$
\begin{equation}\aligned
\lim_{r\rightarrow\infty}\sup_{y,z\in B_{r_1}^r} \left|\left(\tilde{u}_r(y)-\tilde{u}_r(z)\right)^2-|Du|_0^2\cdot\left(d_r(y,z)^2-d_r(y_0,z_0)^2\right)\right|=0,
\endaligned
\end{equation}
where $y_0,z_0\in\G_r$ satisfy $d_r(y,y_0)=d_{\G_r}(y)$ and $d_r(z,z_0)=d_{\G_r}(z)$.
\end{lemma}

\begin{proof} We shall use the idea of  the proof of (23.16) in \cite{Fu} to show our Lemma.
For any small $\de>0$,
let $\ep_i(r)$ be a general positive function satisfying $\lim_{r\rightarrow\infty}\ep_i(r)=0$, which depends only on $n,r_1,c,V_{\Si},\de$ for $i=1,2,\cdots$. It is sufficient to show that
$$\left|\left(\tilde{u}_r(y)-\tilde{u}_r(z)\right)^2-|Du|_0^2\cdot\left(d_r(y,z)^2-d_r(y_0,z_0)^2\right)\right|\le\ep_1(r)+\de$$
for any $y,z\in B_{r_1}^r$.
Suppose $d_{\G_r}(y)\ge d_{\G_r}(z)\ge0$ ($d_{\G_r}(y)d_{\G_r}(z)\le0$ is similar). Let $l_1:\ [0,d_{\G_r}(y)]\rightarrow r^{-2}\Si$ be a normal minimal geodesic joining $y_0$ to $y$, and $l_2:\ [0,d_{\G_r}(z)]\rightarrow r^{-2}\Si$ be a normal minimal geodesic joining $z_0$ to $z$. When $d_{\G_r}(z)\le t\le d_{\G_r}(y)$, we set $l_2(t)=z$ for convenience. Let $Q(t)=d_r(l_1(t),l_2(t))$ and $\g_t=\g_{t,r}:\ [0,Q(t)]\rightarrow r^{-2}\Si$ be a normal minimal geodesic joining $l_2(t)$ to $l_1(t)$.

Let $h_t(s)=\tilde{u}_r(\g_t(s))$, then
\begin{equation}\aligned
\left|\f{d^2h_t}{ds^2}\right|=\left|\mathrm{Hess}_{\tilde{u}_r}(\dot{\g}_t(s),\dot{\g}_t(s))\right|\le\left|\mathrm{Hess}_{\tilde{u}_r}\right|_r.
\endaligned
\end{equation}
Note that $\g_{t,r_{i_j}}$  converges to a normal minimal geodesic $\tilde{\g}_t$ as $r_{i,j}^{-1}\Si$ converges to a cone $\Si_0$.
Hence due to $\G_{r_{i_j}}^s$ converging to a cone $C\mathcal{Y}_s\subset\Si_0$ and $\tilde{\g}_t$ living on one side of $C\mathcal{Y}_{\f{3t}{4|Du|_0}}$ we obtain $\g_{t,r_{i_j}}(s)\in r^{-2}\Si\setminus B^r_{\f{\ep}{2|Du|_0}}$ for any $0\le s\le Q(t)$ and $t\ge\ep$ if $r$ is sufficiently large. Since we can choose any sequence $r_i$ and then choose a suitable subsequence, we conclude $\g_{t,r}\in r^{-2}\Si\setminus B^r_{\f{\ep}{2|Du|_0}}$ for $t\ge\ep$ and sufficiently large $r$.
Hence combining \eqref{f2} and the Newton-Leibniz formula we have
\begin{equation}\aligned
\left|h_t(Q(t))-h_t(0)-Q(t)\f{dh_t}{ds}(Q(t))\right|\le\ep_2(r)
\endaligned
\end{equation}
for any fixed $t\ge\ep$. Note $h_t(Q(t))=\tilde{u}_r(l_1(t))$ and $h_t(0)=\tilde{u}_r(l_2(t))$, Lemma \ref{ury} implies
\begin{equation}\aligned
\left|h_t(Q(t))-|Du|_0t\right|\le\ep_3(r)\qquad &\mathrm{for}\ 0\le t\le d_{\G_r}(y),\\
\left|h_t(0)-|Du|_0t\right|\le\ep_4(r)\qquad &\mathrm{for}\ 0\le t\le d_{\G_r}(z).\\
\endaligned
\end{equation}
Without loss of generality, we assume $d_{\G_r}(y)>\ep$. So we obtain
\begin{equation}\aligned\label{QtGzy}
\left|Q(t)\f{dh_t}{ds}(Q(t))\right|\le\ep_5(r)\qquad &\mathrm{for}\ \ep\le t< \max\{d_{\G_r}(z),\ep\},\\
\left|Q(t)\f{dh_t}{ds}(Q(t))-|Du|_0\left(t-d_{\G_r}(z)\right)\right|\le\ep_6(r)\qquad &\mathrm{for}\  \max\{d_{\G_r}(z),\ep\}\le t\le d_{\G_r}(y).
\endaligned
\end{equation}
Analogously we get
\begin{equation}\aligned\label{Q0Qtep}
\left|Q(t)\f{dh_t}{ds}(0)\right|+\left|Q(t)\f{dh_t}{ds}(Q(t))\right|\le2\ep_5(r)\qquad \mathrm{for}\ \ep\le t\le \max\{d_{\G_r}(z),\ep\}.
\endaligned
\end{equation}
Note initial data $Q(0)=d_r(y_0,z_0)$, then for $0\le t\le d_{\G_r}(z)$ we have
\begin{equation}\aligned\label{Qtdry0z0}
&\left|Q^2(t)-d_r(y_0,z_0)^2\right|\le2\int_0^t\left|Q(s)\f{dQ}{ds}\right|ds\\
=&2\int_0^tQ(s)\left|\left\lan \dot{\g}_s(Q(s)), \dot{l}_1(s)\right\ran_r-\left\lan \dot{\g}_s(0), \dot{l}_2(s)\right\ran_r\right|ds.
\endaligned
\end{equation}
Since
$$\left|\f{dh_t}{ds}\right|_{t=0}+\left|\f{dh_t}{ds}\right|_{t=Q(s)}=\left|\left\lan \dot{\g}_s(0), D^r\tilde{u}_r\Big|_{l_2(s)}\right\ran_r\right|+\left|\left\lan \dot{\g}_s(Q(s)), D^r\tilde{u}_r\Big|_{l_1(s)}\right\ran_r\right|,$$
then combining \eqref{Q0Qtep} and \eqref{Qtdry0z0} gets
\begin{equation}\aligned\label{Du0Q2t}
|Du|_0&\left|Q^2(t)-d_r(y_0,z_0)^2\right|\le2\int_0^tQ(s)\left|\left\lan \dot{\g}_s(Q(s)), |Du|_0\dot{l}_1(s)-D^r\tilde{u}_r\big|_{l_1(s)}\right\ran_r\right|ds\\
+&2\int_0^tQ(s)\left|\left\lan \dot{\g}_s(0), |Du|_0\dot{l}_2(s)-D^r\tilde{u}_r\big|_{l_2(s)}\right\ran_r\right|ds+ 4\ep_5(r)+C\ep,
\endaligned
\end{equation}
where $C$ stands for  a general positive constant.
Lemma \ref{ury} indicates
\begin{equation}\aligned\label{Du0liDru}
\int_0^t\left(|Du|_0-\left\lan \dot{l}_i(s),D^r\tilde{u}_r\big|_{l_i(s)}\right\ran_r\right)ds\le\ep_6(r)
\endaligned
\end{equation}
for $i=1,2$. Then
\begin{equation}\aligned\label{Du0li}
&\left(\int_0^t\left||Du|_0\dot{l}_i(s)-D^r\tilde{u}_r\big|_{l_i(s)}\right|_rds\right)^2
\le t\int_0^t\left||Du|_0\dot{l}_i(s)-D^r\tilde{u}_r\big|_{l_i(s)}\right|_r^2ds\\
=&t\int_0^t\left(|Du|_0^2+\left|D^r\tilde{u}_r(l_i(s))\right|_r^2-2|Du|_0\left\lan \dot{l}_i(s),D^r\tilde{u}_r\big|_{l_i(s)}\right\ran_r\right)ds\\
\le&t\int_0^t\left(\left|D^r\tilde{u}_r(l_i(s))\right|_r^2-|Du|_0^2\right)ds+2t^2|Du|_0\ep_6(r)\le2t^2|Du|_0\ep_6(r)\le\left(\ep_7(r)\right)^2.
\endaligned
\end{equation}
Hence combining \eqref{Du0Q2t} and \eqref{Du0li} we obtain
\begin{equation}\aligned\label{Du0Q2tdry0z0}
|Du|_0&\left|Q^2(t)-d_r(y_0,z_0)^2\right|\le\ep_8(r)+C\ep.
\endaligned
\end{equation}
For $d_{\G_r}(z)\le t\le d_{\G_r}(y)$ we have
\begin{equation}\aligned\label{htQt}
\left|\f{dh_t}{ds}(Q(t))-|Du|_0\f{dQ}{dt}(t)\right|=\left|\left\lan \dot{\g}_t(Q(t)), D^r\tilde{u}_r\big|_{l_1(s)}\right\ran_r-|Du|_0\left\lan \dot{\g}_t(Q(t)), \dot{l}_1(t)\right\ran_r\right|,
\endaligned
\end{equation}
then similar to the argument for $t\le \max\{d_{\G_r}(z),\ep\}$, employing \eqref{Du0liDru} and \eqref{htQt}, and integrating the second inequality in \eqref{QtGzy} give
\begin{equation}\aligned\label{Qdryz}
\left|Q^2(d_{\G_r}(y))-Q^2(d_{\G_r}(z))-\left(d_{\G_r}(y)-d_{\G_r}(z)\right)^2\right|\le\ep_{9}(r)+C\ep,
\endaligned
\end{equation}
where $C$ is a constant. Note $d_r(y,z)=Q(d_{\G_r}(y))$. Combining \eqref{Du0Q2tdry0z0}, \eqref{Qdryz} and Lemma \ref{ury} we have
\begin{equation}\aligned
&\left|\left(\tilde{u}_r(y)-\tilde{u}_r(z)\right)^2-|Du|_0^2\cdot\left(d_r(y,z)^2-d_r(y_0,z_0)^2\right)\right|\\
\le&\left|\left(\tilde{u}_r(y)-\tilde{u}_r(z)\right)^2-|Du|_0^2\cdot\left(Q^2(d_{\G_r}(y))-Q^2(d_{\G_r}(z))\right)\right|\\
&+|Du|_0^2\left|Q^2(d_{\G_r}(z))-d_r(y_0,z_0)^2\right|\\
\le&|Du|_0\left(\ep_8(r)+C\ep\right)+|Du|_0^2\left(\ep_9(r)+C\ep\right)\\
&+\left|\left(\tilde{u}_r(y)-\tilde{u}_r(z)\right)^2-|Du|_0^2\cdot\left(d_{\G_r}(y)-d_{\G_r}(z)\right)^2\right|\\
=&\ep_{10}(r)+C\ep+\left|\left(\tilde{u}_r(y)-\tilde{u}_r(z)\right)^2-\left(|Du|_0\cdot d_{\G_r}(y)-|Du|_0\cdot d_{\G_r}(z)\right)^2\right|\\
\le&\ep_{11}(r)+C\ep.
\endaligned
\end{equation}
Hence we complete the proof.
\end{proof}

For any sequence $r_i\rightarrow\infty$ there is a subsequence $r_{i_j}$ such that $r_{i_j}^{-1}\Si=(\Si,r_{i_j}^{-1}\si,p)$ converges to a metric cone $C\mathfrak{X}$ with vertex $o$ over some smooth manifold $\mathfrak{X}$ in the pointed Gromov-Hausdorff sense. Let $\mathfrak{B}_r$ be the geodesic ball with radius $r$ and centered at $o$ in $C\mathfrak{X}$. Then with \eqref{GHSigma} we get
\begin{equation}\aligned
\lim_{j\rightarrow\infty}d_{GH}\big(\mathfrak{B}_{r_1},B_{(p,0)}({r_1},\G_{r_{i_j}}\times\R)\big)=0.
\endaligned
\end{equation}
By the previous argument, there exists an $(n-1)$-dimensional cone $\mathcal{Y}$ so that $\G_{r_{i_j}}$ converges to $\mathcal{Y}$.
Hence $C\mathfrak{X}=\mathcal{Y}\times\R$, namely, any tangent cone of $\Si$ at infinity splits off a factor $\R$ isometrically.

Let $B^r_s(z)$ be the ball with radius $s$ and centered at $z$ in $r^{-2}\Si=(\Si,r^{-2}\si,p)$.
For any $\ep>0$ and $s>0$, by volume comparison theorem and condition C3), there is a sufficiently large $r_0>0$ such that
$$Vol\big(B_s^{r_i}(z_i)\big)\ge(\omega_n-\ep)s^n$$
for each $z_i\in r_i^{-2}\Si$ with $d_{r_i}(z_i,p)\ge r_0+s$ and $r_i\rightarrow\infty$. Let $\mathfrak{B}_s(z)$ be the geodesic ball with radius $s$ and centered at $z$ in $C\mathfrak{X}$. Taking limit in the above inequality gets
\begin{equation}\aligned\label{VfrakBszep}
Vol\big(\mathfrak{B}_s(z)\big)\ge(\omega_n-\ep)s^n
\endaligned
\end{equation}
for any $z\in C\mathfrak{X}$ with $d_{\infty}(z,o)\ge r_0+s$, where $d_{\infty}$ is the distance function on $C\mathfrak{X}$. Let $\widetilde{\mathfrak{B}}_s(y)$ be the geodesic ball with radius $s$ and centered at $y$ in $C\mathcal{Y}$. Let $z=(y,t_z)\in\mathcal{Y}\times\R$, then \eqref{VfrakBszep} implies
\begin{equation}\aligned\label{IVfrakBszep}
\int_{-s}^sVol\big(\widetilde{\mathfrak{B}}_{\sqrt{s^2-t^2}}(y)\big)dt\ge(\omega_n-\ep)s^n.
\endaligned
\end{equation}
Let $r_0\rightarrow\infty$, and we fix $y$, then $t_z\rightarrow\infty$. So we obtain \eqref{IVfrakBszep} for any $y\in\mathcal{Y}$ and $\ep>0$. Hence
\begin{equation}\aligned\nonumber
\int_{-s}^sVol\big(\widetilde{\mathfrak{B}}_{\sqrt{s^2-t^2}}(y)\big)dt\ge\omega_ns^n,
\endaligned
\end{equation}
which means
\begin{equation}\aligned\label{VfrakBsz}
Vol\big(\mathfrak{B}_{s}(z)\big)\ge\omega_ns^n \qquad \mathrm{for\ any}\ z\in C\mathfrak{X}\ \mathrm{and}\ s>0 .
\endaligned
\end{equation}
Since $\lim_{r\rightarrow\infty}\f{Vol(B_r(x))}{r^n}$ is monotonically nonincreasing, then \eqref{VfrakBsz} implies
\begin{equation}\aligned
Vol\big(B_{s}(z)\big)\ge\omega_ns^n \qquad \mathrm{for\ any}\ z\in \Si\ \mathrm{and}\ s>0.
\endaligned
\end{equation}
By the Bishop volume estimate, $\Si$ is isometric to $\R^n$ (see also the proof of Theorem 0.3 in \cite{C}).

Altogether, we obtain the following Liouville type theorem for minimal
graphic functions with linear growth. This should be compared with the harmonic function theory in \cite{CCM}.
\begin{theorem}\label{SplitLiouville}
Let $u$ be an entire solution to \eqref{u} on a complete Riemannian manifold $\Si$ with conditions C1), C2), C3).
If $u$ has at most linear growth on one side, then $u$ must be a constant unless $\Si$ is isometric to Euclidean space.
\end{theorem}

\section{A Liouville theorem for minimal graphic functions without growth conditions}

Let $G(p,\cdot)$ be the Green function on $\Si^n(n\ge3)$ and $b=G^{\f1{2-n}}$ as before. Now we set
$$\tilde{b}=\left(\f{\omega_n}{V_{\Si}}\right)^{\f1{n-2}}b$$
and define a function $\mathcal{R}$ in $\Si\times\R$ by
$$\mathcal{R}(x,t)=\sqrt{\tilde{b}^2(x)+t^2},\qquad \mathrm{for}\ (x,t)\in\Si\times\R.$$
Then
\begin{equation}\aligned
\De_N\mathcal{R}^2=2n|\na \tilde{b}|^2+2,\qquad |\overline{\na}\mathcal{R}|^2=\f{\tilde{b}^2|\na \tilde{b}|^2+t^2}{\tilde{b}^2+t^2}\le\left(\f{\omega_n}{V_{\Si}}\right)^{\f2{n-2}}.
\endaligned
\end{equation}
Let $\mathbb{B}_r$ be the ball in $\Si\times\R$ with radius $r$ and centered at $(p,0)$.
By the properties \eqref{Deb}\eqref{br}\eqref{Hessb2} of the function $b$ we have
\begin{equation}\aligned
\De_{N}\mathcal{R}^2-2n|\overline{\na}\mathcal{R}|^2=2+2n\f{t^2}{\mathcal{R}^2}\left(|\na \tilde{b}|^2-1\right),
\endaligned
\end{equation}
\begin{equation}\aligned\label{rRnaR}
\limsup_{r\rightarrow\infty}\left(\sup_{\p \mathbb{B}_r}\left|\f {\mathcal{R}}r-1\right|+\sup_{\p \mathbb{B}_r}\Big|\left|\overline{\na} {\mathcal{R}}\right|-1\Big|\right)=0,
\endaligned
\end{equation}
and
\begin{equation}\aligned\label{rHR2}
\limsup_{r\rightarrow\infty}\left(\sup_{\p B_r\times\R}\left|\overline{\mathrm{Hess}}_{\mathcal{R}^2}-2\bar{g}\right|\right)=0.
\endaligned
\end{equation}
Let $\nu$ be the unit normal vector on $M$ as before.
A simple calculation gives
\begin{equation}\aligned\label{DeR2}
\De\mathcal{R}^2=&\De_N\mathcal{R}^2-\overline{\mathrm{Hess}}_{\mathcal{R}^2}(\nu,\nu)=2n|\na \tilde{b}|^2+2-\overline{\mathrm{Hess}}_{\mathcal{R}^2}(\nu,\nu)\\
=&2n|\overline{\na}\mathcal{R}|^2+2+2n\f{t^2}{\mathcal{R}^2}\left(|\na \tilde{b}|^2-1\right)-\overline{\mathrm{Hess}}_{\mathcal{R}^2}(\nu,\nu).
\endaligned
\end{equation}
Since $M$ is an entire graph,
$$\lim_{r_i\rightarrow\infty}\left(\inf\left\{d(x)\Big|\ \mathrm{there\ is\ a}\ t\in\R \mathrm{\ such\ that\ } (x,t)\in M\setminus \mathbb{B}_{\sqrt{r_i}}\right\}\right)=\infty.$$
Combining \eqref{rHR2}\eqref{DeR2}\eqref{br} there exists a sequence $\de_i\rightarrow0^+$ such that on $M\setminus \mathbb{B}_{\sqrt{r_i}}$ we have
\begin{equation}\aligned
\left|\De\mathcal{R}^2-2n|\overline{\na}\mathcal{R}|^2\right|\le2\de_i|\overline{\na}\mathcal{R}|^2.
\endaligned
\end{equation}

Obviously $\Si\times\R$ has nonnegative Ricci curvature. Therefore,  $Vol(\p \mathbb{B}_r)\le|\S^n|r^n$, where $|\S^n|$ is the volume of $n$-dimensional unit sphere in $\R^{n+1}$. Since $M$ is an area-minimizing hypersurface in $\Si\times\R$ by Lemma 2.1 in \cite{DJX}, then
$$Vol(M\cap\mathbb{B}_r)\le\f12Vol(\p\mathbb{B}_r)\le\f{|\S^n|}2r^n.$$
With \eqref{rRnaR},
$$r^{-n}\int_{M\cap\{\mathcal{R}\le r\}}|\overline{\na}\mathcal{R}|^2d\mu$$
is uniformly bounded for any $r\in(0,\infty)$, and so,  there exists a sequence $r_i\rightarrow\infty$ such that
$$\limsup_{r\rightarrow\infty}\left(r^{-n}\int_{M\cap\{\mathcal{R}\le r\}}|\overline{\na}\mathcal{R}|^2d\mu\right)
=\lim_{r_i\rightarrow\infty}\left(r_i^{-n}\int_{{M\cap\{\mathcal{R}\le r_i\}}}|\overline{\na}\mathcal{R}|^2d\mu\right).$$

With  the proof of Lemma 5.2 in \cite{DJX}, we obtain the following Lemma.
\begin{lemma}\label{limfdR}
There is a sequence $\de_i\rightarrow0^+$ such that for any constants $K_2>K_1>0$ and $\ep\in(0,1)$ and any bounded Lipschitz function $f$ on $N\setminus \mathbb{B}_1$ we have
\begin{equation}\aligned
&\limsup_{i\rightarrow\infty}\left|\left(\f{\de_i}{K_2r_i}\right)^{n}\int_{M\cap\{\mathcal{R}\le \f{K_2r_i}{\de_i}\}}f|\overline{\na}\mathcal{R}|^2-\left(\f{\de_i}{K_1r_i}\right)^{n}\int_{M\cap\{\mathcal{R}\le \f{K_1 r_i}{\de_i}\}}f|\overline{\na}\mathcal{R}|^2\right|\\
\le&C\ep^n\sup_{N\setminus \mathbb{B}_1}|f|+\limsup_{i\rightarrow\infty}\int_{\f{K_1 r_i}{\de_i}}^{\f{K_2r_i}{\de_i}}\left(s^{-n-1}\int_{M\cap\{\f{\ep K_1r_i}{\de_i}<\mathcal{R}\le s\}}\mathcal{R}\na f\cdot\na\mathcal{R}\right)ds.
\endaligned
\end{equation}
\end{lemma}

There is a subsequence $\{\ep_i\}$ of $\{\de_i^2r_i^{-2}\}$ converging to zero such that  $\ep_i\Si=(\Si,\ep_i \si,p)$ converges to a metric cone $(\Si_{\infty},d_{\infty})$ with vertex $o$ in the measured Gromov-Hausdorff sense. Denote $\ep_i=\de_i^2r_i^{-2}$ for simplicity. The cone $\Si_{\infty}$ is over some $(n-1)$-dimensional smooth compact manifold $X$ with $C^{1,\a}$ Riemannian metric and Diam $X\le\pi$, namely, $\Si_{\infty}=CX\triangleq\R^+\times_{\r}X$.


Let $B^i_r(x)$ be the geodesic ball with radius $r$ and centered at
$x$ in $(\Si,\ep_i\si)$, and $\mathcal{B}_r(x)$ be the geodesic ball
with radius $r$ and centered at $x$ in $\Si_{\infty}$. In particular, $X=\p \mathcal{B}_1(o)$. Note that for convenience our definitions of $B^i_r(x)$ are different from the previous ones in section 4.
Let $\mathbb{B}^i_r(x)$ be the geodesic ball with radius $r$ and centered at $x$ in $(\Si\times\R,\ep_i(\si+dt^2))$, and $\widetilde{\mathcal{B}}_r(x)$ be the geodesic ball with radius $r$ and centered at $x$ in $\Si_{\infty}\times\R$.

Let $\ep_iM=(M,\ep_i g)$ and $D^i_r(x)=\ep_iM\cap\mathbb{B}^i_r(x)$. We always omit $x$ in $D^i_r(x)$ (or $\mathbb{B}^i_r(x), \mathcal{B}_r(x), \widetilde{\mathcal{B}}_r(x)$) if $x=p$ (or $x=(p,0), o, (o,0))$ respectively, for simplicity. Clearly, $\ep_iM$ is still a minimal graph in $(\Si\times\R,\ep_i(\si+dt^2))$.
\begin{lemma}
There exists a subsequence $\{\ep_{i_j}\}\subset\{\ep_i\}$ such that $\ep_{i_j}M$ converges to an area-minimizing cone $T=CY\triangleq\R^+\times_{\r}Y$ in $\Si_{\infty}\times\R$, where $Y\in \p\widetilde{\mathcal{B}}_1(o)$ is an $(n-1)$-dimensional Hausdorff set.
\end{lemma}
\begin{proof}
For any fixed $r>1$ let $\Upsilon_i:\ \left(\overline{\mathcal{B}}_{r+1}\setminus \mathcal{B}_{\f1{2r}}\right)\times\R\rightarrow \ep_i\Si\times\R$ be a mapping defined by $\Upsilon_i(x,t)=(\Phi_i(x),t)\in\ep_i\Si\times\R$, where $\Phi_i$ is a diffeomorphism from $\overline{\mathcal{B}}_{r+1}\setminus \mathcal{B}_{\f1{2r}}$ to $\Phi_i(\overline{\mathcal{B}}_{r+1}\setminus \mathcal{B}_{\f1{2r}})\subset\ep_i\Si$ such that $\Phi_i^*(\ep_i \si)$ converges as $i\rightarrow\infty$ to $\si_{\infty}$ in the $C^{1,\a}$-topology on $\overline{\mathcal{B}}_{r+1}\setminus \mathcal{B}_{\f1{2r}}$.
Thus $\Upsilon_i^*(\ep_i \si+\ep_idt^2)$ converges as
$i\rightarrow\infty$ to $\sigma_{\infty}+dt^2$ in the
$C^{1,\a}$-topology on $\widetilde{\mathcal{B}}_{r+1}\setminus
\widetilde{\mathcal{B}}_{\f1{2r}}$. By compactness of currents (see
\cite{AK}, \cite{LY}, \cite{S} or \cite{DJX}), there is a subsequence of $\ep_{i_j}$ such that
$$\Upsilon_{i_j}^{-1}\left(\ep_{i_j}M\bigcap\left(B^{i_j}_r\setminus B^{i_j}_{\f1r}\right)\times\R\right)\rightharpoonup T\qquad \mathrm{as}\ j\rightarrow\infty,$$
where $T$ is an integral-rectifiable current in
$\Si_{\infty}\times\R$. By choosing a diagonal sequence, we can assume that the above limit holds for any $r>1$. For convenience, we still write $\ep_i$  instead of $\ep_{i_j}$.

Let $\Om_0$ be an arbitrary bounded domain in $T$, and $W_0$ be an arbitrary bounded set with induced metric in $\Si_\infty\times\R$ with $\p\Om_0=\p W_0$. There is a constant $R>0$ such that $\Om_0\cup W_0\subset \widetilde{\mathcal{B}}_R$.
For any small $\de>0$ let $\Om_i=\Upsilon_i(\Om_0\setminus(\mathcal{B}_{\de}\times[-R,R]))\subset \ep_iM$ and $W_i=\Upsilon_i(W_0\setminus(\mathcal{B}_{\de}\times[-R,R]))$ with induced metrics in $\ep_iN$. Then there exists $U_0\subset\p(\mathcal{B}_{\de}\times[-R,R])$ (possibly empty) such that $\p\Om_i=\p(W_i\cup U_i)$ with $U_i=\Upsilon_i(U_0)\subset\ep_iN$. Since $\ep_iM$ is an area-minimizing hypersurface in $\ep_iN$, then
\begin{equation}\aligned
H^n(\Om_0\setminus(\mathcal{B}_{\de}\times[-R,R]))=&\lim_{i\rightarrow\infty}H^n(\Om_i)\le\lim_{i\rightarrow\infty}H^n(W_i\cup U_i)\\
\le&\lim_{i\rightarrow\infty}H^n(W_i)+\lim_{i\rightarrow\infty}H^n(U_i)
=H^n(W_0\setminus\widetilde{\mathcal{B}}_{\de})+H^n(U_0)\\
\le& H^n(W_0)+H^n(\p(\mathcal{B}_{\de}\times[-R,R])).
\endaligned
\end{equation}
Let $\de\rightarrow0$ to obtain
$$H^n(\Om_0)\le H^n(W_0).$$
Namely, $T$ is an area-minimizing set in $\Si_{\infty}\times\R$.

For any $f\in C^1(\p\widetilde{\mathcal{B}}_1)$, we could extend $f$ to $\Si_{\infty}\times\R\setminus\{(o,0)\}$ by defining
$$f(\r\th)=f(\th)$$
for any $\r>0$ and $\th\in\p\widetilde{\mathcal{B}}_1$. Let $\Pi_i$ be the map of rescaling from $(N,\si+dt^2)$ to $\ep_iN=(N,\ep_i\si+\ep_idt^2)$.
Set $U_{s}=B_s\times\R$ for $s>0$.
Note $\ep_i=\de_i^2r_i^{-2}$, then similar to the proof of (4.12) and (4.13) in \cite{DJX}, for any $K_2>K_1>0$ we have
\begin{equation}\aligned
\limsup_{i\rightarrow\infty}\sup_{\mathbb{B}_{\f{K_2r_i}{\de_i}}\setminus U_{\f{\ep K_1r_i}{\de_i}}}
\left|\left\lan\overline{\na}(f\circ\Upsilon_i^{-1}\circ\Pi_i),\overline{\na}\mathcal{R}^2\right\ran\right|=0,
\endaligned
\end{equation}
and
\begin{equation}\aligned
\limsup_{i\rightarrow\infty}\sup_{\mathbb{B}_{\f{K_2r_i}{\de_i}}\setminus U_{\f{\ep K_1r_i}{\de_i}}}
\left(\mathcal{R}\left|\overline{\na}(f\circ\Upsilon_i^{-1}\circ\Pi_i)\right|\right)<\infty,
\endaligned
\end{equation}
Now we can extend the function $f\circ\Upsilon_i^{-1}\circ\Pi_i$ to a
uniformly bounded function $F_i$ in
$\mathbb{B}_{\f{K_2r_i}{\de_i}}\setminus U_{\f{\ep K_1r_i}{\de_i}}$
with $F_i=f\circ\Upsilon_i^{-1}\circ\Pi_i$ on
$\mathbb{B}_{\f{K_2r_i}{\de_i}}\setminus U_{\f{\ep
    K_1r_i}{\de_i}}$. Obviously, we can extend $F_i$ to a $C^1$-function on $\mathbb{B}_{\f{K_2r_i}{\de_i}}\cap U_{\f{\ep K_1r_i}{\de_i}}$ with $|F_i|\le 2|f_0|_{C^0(\p \widetilde{\mathcal{B}}_1)}$.

Note $\mathcal{R}^2(x,t)=\tilde{b}^2(x)+t^2$ for any $(x,t)\in\Si\times\R$.
Due to the proof of Lemma 5.3 in \cite{DJX}, it is sufficient to show that there is a sequence $\tau_i\in[\ep,2\ep]$ so that
\begin{equation}\aligned\label{IFiR2}
&\limsup_{i\rightarrow\infty}\int_{\f{K_1 r_i}{\de_i}}^{\f{K_2r_i}{\de_i}}\left(\f1{s^{n+1}}\int_{M\cap\{\f{\tau_i K_1 r_i}{\de_i}<\mathcal{R}\le s\}\cap \left(\{\tilde{b}\le\f{\tau_i K_1r_i}{\de_i}\}\times\R\right)}\mathcal{R}\na F_i\cdot\na\mathcal{R}\right)ds<C\ep\\
\endaligned
\end{equation}
for some absolute constant $C>0$.

We show  (\ref{IFiR2}) by the following consideration.

Let $\Om_{s,i,\tau}=M\cap\{\f{\tau K_1 r_i}{\de_i}<\mathcal{R}\le s\}\cap \left(\{\tilde{b}\le\f{\tau K_1r_i}{\de_i}\}\times\R\right)$ for $s\in\left(\f{K_1 r_i}{\de_i},\f{K_2 r_i}{\de_i}\right)$. Integrating by parts implies
\begin{equation}\aligned
&\int_{\Om_{s,i,\tau}}\na F_i\cdot\na\mathcal{R}^2+\int_{\Om_{s,i,\tau}}F_i\De\mathcal{R}^2=\int_{\Om_{s,i,\tau}}\mathrm{div}_M\left(F_i\na\mathcal{R}^2\right)\\
=&\int_{\p\Om_{s,i,\tau}}F_i\lan\na\mathcal{R}^2,\nu_{\p\Om_{s,i,\tau}}\ran\le\int_{\p\Om_{s,i,\tau}}F_i|\na\mathcal{R}^2|\le C_1s\int_{\p\Om_{s,i,\tau}}1
\endaligned
\end{equation}
for some absolute constant $C_1$. Recall $|\overline{\na}\mathcal{R}|\le\left(\f{\omega_n}{V_{\Si}}\right)^{\f1{n-2}}$. It is easy to see that
\begin{equation}\aligned
&\left(\f{V_{\Si}}{\omega_n}\right)^{\f1{n-2}}\f{K_1r_i}{\de_i}\int_\ep^{2\ep}\left(\int_{\p\Om_{s,i,\tau}}1\right)d\tau
\le\f{K_1r_i}{\de_i}\int_\ep^{2\ep}\left(\int_{\p\Om_{s,i,\tau}}\f1{|\na\mathcal{R}|}\right)d\tau\\
\le& Vol\left(M\cap\{\mathcal{R}\le s\}\cap \left(\left\{\f{\ep K_1r_i}{\de_i}\le\tilde{b}\le\f{2\ep K_1r_i}{\de_i}\right\}\times\R\right)\right)\\
&+Vol\left(M\cap\left\{\f{\ep K_1r_i}{\de_i}\le\mathcal{R}\le\f{2\ep K_1r_i}{\de_i}\right\}\right)\\
\le& Vol\left(\p\left(\{\mathcal{R}\le s\}\cap \left(\left\{\f{\ep K_1r_i}{\de_i}\le\tilde{b}\le\f{2\ep K_1r_i}{\de_i}\right\}\times\R\right)\right)\right)\\
&+Vol\left(\p\left(\left\{\f{\ep K_1r_i}{\de_i}\le\mathcal{R}\le\f{2\ep K_1r_i}{\de_i}\right\}\right)\right)\\
\le& C_2s\left(\f{\ep K_1r_i}{\de_i}\right)^{n-1}
\endaligned
\end{equation}
for some absolute constant $C_2$. Hence, for every $i$ there exists a $\tau_i\in[\ep,2\ep]$ such that
\begin{equation}\aligned
\int_{\p\Om_{s,i,\tau_i}}1\le C_3s\left(\f{\ep K_1r_i}{\de_i}\right)^{n-2}.
\endaligned
\end{equation}
Through a simple calculation we have
\begin{equation}\aligned
&\left|\int_{\Om_{s,i,\tau_i}}F_i\De\mathcal{R}^2\right|\le C_4Vol\left(\Om_{s,i,\tau_i}\right)\\
\le& C_4Vol\left(\p\left(\left\{\f{\tau K_1 r_i}{\de_i}<\mathcal{R}\le s\right\}\bigcap \left(\left\{\tilde{b}\le\f{\tau K_1r_i}{\de_i}\right\}\times\R\right)\right)\right)\\
\le& C_5s\left(\f{\ep K_1r_i}{\de_i}\right)^{n-1}.
\endaligned
\end{equation}
Hence
\begin{equation}\aligned
&\int_{\Om_{s,i,\tau_i}}\na F_i\cdot\na\mathcal{R}^2\le-\int_{\Om_{s,i,\tau_i}}F_i\De\mathcal{R}^2+ C_1s\int_{\p\Om_{s,i,\tau_i}}1\le C_6s^2\left(\f{\ep K_1r_i}{\de_i}\right)^{n-2},
\endaligned
\end{equation}
which implies that \eqref{IFiR2} holds. Then, as in Lemma 5.3 in \cite{DJX}, we can show that
\begin{equation}
\f1{K_2^{n}}\int_{T\cap\widetilde{\mathcal{B}}_{K_2}}f=\f1{K_1^{n}}\int_{T\cap\widetilde{\mathcal{B}}_{K_1}}f
\end{equation}
for  arbitrary $f.$  This means $T$ is a cone in $\Si_{\infty}\times\R$.  Therefore, we complete the proof.
\end{proof}

Set $s_i=(\ep_i)^{-\f12}$ for convenience. 
The definitions of $D^{s_i},\tilde{u}_{s_i},\lan\cdot,\cdot\ran_{s_i},|\cdot|_{s_i}$ are as in section 4. We define $\r_\infty(x)=d_\infty(o,x)$ and $\bar{\r}_\infty(x)=\bar{d}_\infty(o,x)$ as distance functions on $\Si_\infty$ and $N_\infty$, respectively. Let $\r_i(x)$ be the distance function on $\ep_i\Si$ from $o$ to $x$, and $\bar{\r}_i(z)$ be the distance function on $\ep_i(\Si\times\R)$ from $(o,0)$ to $z$.

For each $x\in\ep_i \Si$ there is a minimal normal geodesic $\g^i_x$
from $p$ to $x$ such that  $D^{s_i}\r_i(x)=\dot{\g^i_x}$. When
$\ep_i=1$, we define $D\r(x)$ corresponding to the normal geodesic
$\dot{\g_x}$. Hence $D^{s_i}\r_i(x)$ depends on the choice of $\g^i_x$. Note that $\r_i(x)$ is just a Lipschitz function on $\ep_i\Si$, but the
definition of $D^{s_i}\r_i(x)$ is equivalent to the common one if $\r_i$ is $C^1$ at the considered point.

Let $T_i=\ep_iM\bigcap\left((B^i_2\setminus B^i_{\ep})\times\left[\f1\de,\f1\de+\f2\ep\right]\right)$ for $\de>0$ and a 'bad' set
\begin{equation}\aligned\label{EiTi}
E_i\triangleq &\left\{z=(x,t)\in T_i\bigg|\ |D^{s_i}\tilde{u}_{s_i}(x)|_{s_i}\le\f1\ep\ \mathrm{or}\ \left| \left(D^{s_i}\tilde{u}_{s_i}(x)\right)^T\right|_{s_i}^2\le\left(1-\ep^4\right)\left|D^{s_i}\tilde{u}_{s_i}(x)\right|_{s_i}^2\right\},
\endaligned
\end{equation}
where $\xi^T=\xi-\left\lan \xi,D^{s_i}\r_i\right\ran_{s_i} D^{s_i}\r_i$ for any local vector field $\xi$ on $\ep_i\Si$.
\begin{lemma}\label{HnEiepn}
Suppose $\sup_\Si|Du|=\infty$.
For any $\ep>0$ there is a sufficiently small $\de_0$ such that for any $0<\de<\de_0$ there is a sufficiently large $i_0$ so that for $i\ge i_0$ we have
$$H^n(E_i)<\ep^{n}.$$
\end{lemma}
\begin{proof}
Let $U_j$ be the subgraph of $\tilde{u}_{s_i}$ in $\ep_i\Si\times\R$ defined by
$$\{(x,t)\in\ep_i\Si\times\R|\ t<\tilde{u}_{s_i}(x)\}.$$
By Rellich theorem, for any compact $K\subset\Si_\infty$ there is a subsequence of the
characteristic functions $\chi_{_{U_j}}$ converging to
$\chi_{_{\widetilde{U}}}$ in $L^1(K)$ up to a diffeomorphism (see Proposition 16.5 in
\cite{Gi} for the Euclidean case). Clearly, $\widetilde{U}$ can be represented as a subgraph of some generalized function $\mathfrak{u}$ (possibly equal to $\pm\infty$ somewhere) in $\Si_\infty\times\R$, namely,
$$\widetilde{U}=\{(x,t)\in\Si_\infty\times\R|\ t<\mathfrak{u}(x)\}.$$
Note $T=\p\widetilde{U}$.
$\mathfrak{u}(x)$ is a homogeneous function of degree 1 in $\Si_\infty\setminus\{o\}$ as $T$ is a cone through $o$.

If $\sup_\Si|Du|=\infty$, $T$ contains a half line $\{o\}\times(0,\infty)$ or $\{o\}\times(-\infty,0)$. Without loss generality, we assume $\{o\}\times(0,\infty)\subset T$.
Now we define a set $P$ by $\{x\in\Si_\infty|\ \mathfrak{u}(x)=+\infty\}$.
Since $T$ is a cone through the point $o$, then $P$ is also a cone through $o$ in $\Si_\infty$. In fact, for any $x\in P$ we have $\mathfrak{u}(x)=+\infty$. In particular, there is $(r,\th)\in\R^+\times_{\r}X=CX$ so that $x=(r,\th)$. Then $\mathfrak{u}(tx)=+\infty$ for $tx=(tr,\th)$, which means that $P$ is a cone.

For any $z\in
T\bigcap\left((\mathcal{B}_2\setminus\mathcal{B}_{\ep})\times\{s\}\right)$,
the slope of the line connecting $z$ and $o$ becomes larger and larger as $s$ increases to infinity. Hence for any $0<\ep<1$ we have
\begin{equation}\aligned\label{TB2Bepde}
&\lim_{\de\rightarrow0}d_{GH}\left(T\bigcap\left((\mathcal{B}_2\setminus\mathcal{B}_{\ep})\times\left[\f1\de,\f1\de+\f2\ep\right]\right),
\Big(\p P\cap(\mathcal{B}_2\setminus\mathcal{B}_\ep)\Big)\times\left[0,\f2\ep\right]\right)=0.
\endaligned
\end{equation}
Combining \eqref{2nfDsbq}-\eqref{mean_q} and that $M$ is area-minimizing, it is clear that there is a constant $\varepsilon_0>0$ depending only on $\Si$ so that
$$\varepsilon_0r^n\le\int_{M\cap \mathbb{B}_r(z)}1\le\f12Vol(\p\mathbb{B}_r(z))\le\f{|\S^n|}2r^n\qquad \mathrm{for\ any}\ r>0,\ z\in M.$$
Since $\ep_iM\rightharpoonup T$, for any $r>0$ and $y\in T$ we have
$$\varepsilon_0r^n\le\int_{T\cap \mathcal{B}_r(y)}1\le\f{|\S^n|}2r^n.$$
Due to \eqref{TB2Bepde} and $\{o\}\times(0,\infty)\subset T$, it is not hard to see that
\begin{equation}\aligned\label{Hn-2pP}
H^{n-1}\Big(\p P\cap(\mathcal{B}_2\setminus\mathcal{B}_\ep)\Big)>0\qquad \mathrm{and}\qquad H^{n-2}\Big(\p P\cap\p\mathcal{B}_1\Big)>0
\endaligned
\end{equation}
as $P$ is a cone through $o$.

Let $\Phi_i:\ \overline{\mathcal{B}}_{\f 2\de}\setminus \mathcal{B}_{\ep}\rightarrow\Phi(\overline{\mathcal{B}}_{\f 2\de}\setminus \mathcal{B}_{\ep})\subset\ep_i\Si$ be a diffeomorphism such that $\Phi_i^*(\ep_i \si)$ converges as $i\rightarrow\infty$ to $\si_{\infty}$ in the $C^{1,\a}$-topology on $\overline{\mathcal{B}}_{\f2\de}\setminus \mathcal{B}_{\ep}$, and $\Upsilon_i(x,t)=(\Phi_i(x),t)$ for any $x\in\overline{\mathcal{B}}_{\f 2\de}\setminus \mathcal{B}_{\ep}$. Note that $\Phi_i$ and $\Upsilon_i$ depend on $\de$.
Obviously, $\lim_{i\rightarrow\infty}\r_i\circ\Phi_i=\r_\infty$ in $\overline{\mathcal{B}}_{\f 2\de}\setminus \mathcal{B}_{\ep}$ and $\lim_{i\rightarrow\infty}\bar{\r}_i\circ\Upsilon_i=\bar{\r}_\infty$ in $\left(\mathcal{B}_{\f 2\de}\setminus \mathcal{B}_{\ep}\right)\times\left(-\f2\de,\f2\de\right)$.
Since $\ep_iM\bigcap\left(\mathbb{B}^i_{\f 4\de}\setminus
  \mathbb{B}^i_{\ep}\right)$ converges to
$T\bigcap\left(\widetilde{\mathcal{B}}_{\f 4\de}\setminus
  \widetilde{\mathcal{B}}_{\ep}\right)$ in the varifold sense, for any compact set $\widetilde{K}\in \widetilde{\mathcal{B}}_{\f 4\de}\setminus \widetilde{\mathcal{B}}_{\ep}$ we have
\begin{equation}\aligned\label{epiMUpK}
0=\lim_{i\rightarrow\infty}\left(\ep_iM\llcorner\Upsilon_i(\widetilde{K})\right)(\bar{\omega}^*\circ\Upsilon_i^{-1})
=\lim_{i\rightarrow\infty}\int_{\ep_iM\cap\Upsilon_i(\widetilde{K})}\lan\bar{\omega}^*\circ\Upsilon_i^{-1},\nu_i\ran d\mu_i,
\endaligned
\end{equation}
where $\bar{\omega}^*$ is the dual form of $\f{\p}{\p\bar{\r}_\infty}$ in $TN_\infty$.
For any compact set $K\subset (\mathcal{B}_2\setminus\mathcal{B}_{\ep})\times\left[\f1\de,\f1\de+\f2\ep\right]$, we let $K_s=\{x|\ (x,s)\in K\}$ be a slice of $K$ for $s\in\left[\f1\de,\f1\de+\f2\ep\right]$.

Since $\ep_iM\rightharpoonup T$, \eqref{TB2Bepde} implies
\begin{equation}\aligned\label{HTiDuC}
\lim_{i\rightarrow\infty}H^n\left(\left\{(x,t)\in T_i\Big|\ |D^{s_i}\tilde{u}_{s_i}(x)|_{s_i}\le \f1{\de}\right\}\right)=0.
\endaligned
\end{equation}
Thus from \eqref{epiMUpK} we obtain
\begin{equation}\aligned
\limsup_{i\rightarrow\infty}\left|\int_{\f1\de}^{\f1\de+\f2\ep}\left(\int_{\ep_iM\cap\Phi_i(K_s)}\left\lan\omega^*\circ\Phi_i^{-1},
\f{D^{s_i}\tilde{u}_{s_i}}{\sqrt{1+|D^{s_i}\tilde{u}_{s_i}|_{s_i}^2}}\right\ran_{s_i}\right) ds\right|\le C\f{\de}{\ep}
\endaligned
\end{equation}
for some constant $C$, where $\omega^*$ is the dual form of $\f{\p}{\p\r_\infty}$ in $T\Si_\infty$. Note that $D^{s_i}\r_i\rightarrow\f{\p}{\p\r_\infty}$ as $i\rightarrow\infty$ on any compact set in $\Si_\infty\setminus\{o\}$ up to a diffeomorphism $\Phi_i$, then it follows that
\begin{equation}\aligned\label{1de2epKs}
\limsup_{i\rightarrow\infty}\left|\int_{\f1\de}^{\f1\de+\f2\ep}\left(\int_{\ep_iM\cap\Phi_i(K_s)}
\f{\left\lan D^{s_i}\r_i,D^{s_i}\tilde{u}_{s_i}\right\ran_{s_i}}{\sqrt{1+|D^{s_i}\tilde{u}_{s_i}|_{s_i}^2}}\right) ds\right|\le C\f{\de}{\ep}
\endaligned
\end{equation}
Together with \eqref{HTiDuC} and \eqref{1de2epKs}, we complete the proof.
\end{proof}

Since $M$ is a minimal graph in $\Si\times\R$,  it is stable, and $\ep_iM$ is also a stable minimal hypersurface in $\ep_i(\Si\times\R)$. Let $B^i$ be the second fundamental form of $\ep_iM$ in $\ep_iN$ and $Ric_{\ep_iN}$ be Ricci curvature of $\ep_iN$. For any Lipschitz function $\varphi$ with compact support in $\ep_iM$ we have
\begin{equation}\aligned\label{stableM}
\int_{\ep_iM}\left(|B^i|^2+Ric_{\ep_iN}(\nu_i,\nu_i)\right)\phi^2\le\int_{\ep_iM}|\na^i\phi|^2,
\endaligned
\end{equation}
where $\na^i$ is the Levi-Civita connection of $\ep_iM$.

Now we suppose that there exists sufficiently large $r_0>0$ such that
the non-radial Ricci curvature of $\Si$ satisfies
\begin{equation}\aligned\label{Riccond}
\inf_{\p B_r}Ric_{\Si}\left(\xi^T,\xi^T\right)\ge\f{\k}{r^2}>0
\endaligned
\end{equation}
for all $r\ge r_0$ and $n\ge3$, where $\xi$ is a local vector field on $\Si$ with $\xi^T=\xi-\left\lan \xi,D\r\right\ran D\r$ and $|\xi^T|=1$. The definition of $D\r$ is as before.
\begin{lemma}\label{kn-3}
If $\sup_\Si|Du|=\infty$, then $\k$ in \eqref{Riccond} satisfies $\k\le\f{(n-3)^2}4.$
\end{lemma}
\begin{proof}
By re-scaling we get
$$\inf_{\p B_r^i}Ric_{\ep_i\Si}\left(\e^T,\e^T\right)\ge\f{\k}{r^2}>0$$
for all $r\ge \sqrt{\ep_i}r_0$, where $\e\in\G(T(\ep_iN))$, $\e^T=\e-\left\lan \e,D^{s_i}\r_i\right\ran_{s_i} D^{s_i}\r_i$ with $|\e^T|_{s_i}=1$.
Noting that conditions C1) and C3)  are both invariant under scaling, we obtain
\begin{equation}\aligned\label{RicepiNnu}
&Ric_{\ep_iN}(\nu_i,\nu_i)=\f1{1+|D^{s_i}\tilde{u}_{s_i}|_{s_i}^2}Ric_{\ep_i\Si}(D^{s_i}\tilde{u}_{s_i},D^{s_i}\tilde{u}_{s_i})\\
\ge& \f1{1+|D^{s_i}\tilde{u}_{s_i}|_{s_i}^2}\Big(Ric_{\ep_i\Si}\left((D^{s_i}\tilde{u}_{s_i})^T,(D^{s_i}\tilde{u}_{s_i})^T\right)\\
&\qquad\qquad\qquad\quad +2\left\lan D^{s_i}\tilde{u}_{s_i},D^{s_i}\r_i\right\ran_{s_i} Ric_{\ep_iN}\left((D^{s_i}\tilde{u}_{s_i})^T,D^{s_i}\r_i\right)\Big)\\
\ge& \f1{1+|D^{s_i}\tilde{u}_{s_i}|_{s_i}^2}\Big(\k'\left|(D^{s_i}\tilde{u}_{s_i})^T\right|_{s_i}^2
-c'\left|(D^{s_i}\tilde{u}_{s_i})^T\right|_{s_i}\left\lan D^{s_i}\tilde{u}_{s_i},D^{s_i}\r_i\right\ran_{s_i}\Big)\r_i^{-2}\\
\endaligned
\end{equation}
for some absolute constant $c'>0$.

Let $\e$ be the Lipschitz function on $\ep_i\Si$ defined by
$$\eta(x)=\big(\r_i(x)\big)^{\f{3-n}2}\sin\left(\pi\f{\log\r_i(x)}{\log\ep}\right)$$
in $B_1^i\setminus B_{\ep}^i$ and $\eta=0$ in other places.
Let $\tau$ be a Lipschitz function on $\R$ satisfying $\tau\equiv1$ on $\left[\f 1\de+1,\f 1\de+\f2\ep-1\right]$, $\tau(t)\equiv0$ for $t\in\left(-\infty,\f1\de\right]\bigcup[\f1\de+\f2\ep,\infty)$, and $|\tau'|\le1$.

For $z=(x,t)\in\Si\times\R$ set $\varphi(z)=\eta(x)\tau(t)$. Let
$\bn^i$ be the Levi-Civita connection of $\ep_iN$. Then
\begin{equation}\aligned\label{vetau}
&\int_{\ep_iM}Ric_{\ep_iN}(\nu_i,\nu_i)\varphi^2\le\int_{\ep_iM}|\overline{\na}^i\varphi|^2
=\int_{\ep_iM}\left(|D^{s_i}\eta|_{s_i}^2\tau^2+\eta^2|\tau'|^2\right)\\
\le&\int_{\f1{\de}}^{\f 1\de+\f2\ep}\left(\int_{\ep_iM\cap\ep_i\Si\times\{t\}}|D^{s_i}\eta|_{s_i}^2\right)dt+\left(\int_{\f 1\de}^{\f 1\de+1}+\int^{\f1\de+\f 2\ep}_{\f 1\de+\f2\ep-1}\right)\left(\int_{\ep_iM\cap\ep_i\Si\times\{t\}}\eta^2\right)dt\\
=&\int_{\f1{\de}}^{\f 1\de+\f2\ep}\left(\int_{\ep_iM\cap(B^i_1\setminus B^i_{\ep})\times\{t\}}\left(\f{3-n}2\sin\left(\pi\f{\log\r_i}{\log\ep}\right)+\f{\pi}{\log\ep}\cos\left(\pi\f{\log\r_i}{\log\ep}\right)\right)^2\r_i^{1-n}\right)dt\\
&+\left(\int_{\f 1\de}^{\f 1\de+1}+\int^{\f1\de+\f 2\ep}_{\f 1\de+\f2\ep-1}\right)\left(\int_{\ep_iM\cap(B^i_1\setminus B^i_{\ep})\times\{t\}}\sin^2\left(\pi\f{\log\r_i}{\log\ep}\right)\r_i^{3-n}\right)dt.
\endaligned
\end{equation}
Denote $E_i$ as \eqref{EiTi}. For any $z\in T_i\setminus E_i$ we have $|D^{s_i}\tilde{u}_{s_i}(z)|_{s_i}>\f1\ep$ and
$$\left|\lan D^{s_i}\tilde{u}_{s_i}(z), D^{s_i}\r_i\ran_{s_i}\right|\le\ep^2\left|D^{s_i}\tilde{u}_{s_i}(z)\right|_{s_i}.$$
Note $Ric_{\ep_iN}(\nu_i,\nu_i)=\left(1+|D^{s_i}\tilde{u}_{s_i}|_{s_i}^2\right)^{-1}Ric(D^{s_i}\tilde{u}_{s_i},D^{s_i}\tilde{u}_{s_i})$. Combining \eqref{RicepiNnu} we get
\begin{equation}\aligned
&\int_{\ep_iM}Ric_{\ep_iN}(\nu_i,\nu_i)\varphi^2\ge\int_{\ep_iM\cap\{\f1\de+1\le t\le\f 1\de+\f2\ep-1\}}Ric_{\ep_iN}(\nu_i,\nu_i)\eta^2\\
\ge&\int_{\f1{\de}+1}^{\f 1\de+\f2\ep-1}\bigg(\int_{(\ep_iM\setminus E_i)\cap(B^i_1\setminus B^i_{\ep})\times\{t\}}\f{1}{\r_i^2}\f{1}{1+|D^{s_i}\tilde{u}_{s_i}|_{s_i}^2}\Big(\k\left|(D^{s_i}\tilde{u}_{s_i})^T\right|_{s_i}^2\\
&-c'\left|(D^{s_i}\tilde{u}_{s_i})^T\right|_{s_i}\left\lan D^{s_i}\tilde{u}_{s_i},D^{s_i}\r_i\right\ran_{s_i}\Big)\sin^2\left(\pi\f{\log\r_i}{\log\ep}\right)\r_i^{3-n}\bigg)dt\\
\ge&\int_{\f1{\de}+1}^{\f 1\de+\f2\ep-1}\bigg(\int_{(\ep_iM\setminus E_i)\cap(B^i_1\setminus B^i_{\ep})\times\{t\}}\f{|D^{s_i}\tilde{u}_{s_i}|_{s_i}^2}{1+|D^{s_i}\tilde{u}_{s_i}|_{s_i}^2}\Big(\k(1-\ep^2)\\
&-c'\ep^2\Big)\sin^2\left(\pi\f{\log\r_i}{\log\ep}\right)\r_i^{1-n}\bigg)dt\\
\ge&\f{\k(1-\ep^2)-c'\ep^2}{1+\ep^2}\int_{\f1{\de}+1}^{\f 1\de+\f2\ep-1}\bigg(\int_{(\ep_iM\setminus E_i)\cap(B^i_1\setminus B^i_{\ep})\times\{t\}}\sin^2\left(\pi\f{\log\r_i}{\log\ep}\right)\r_i^{1-n}\bigg)dt.
\endaligned
\end{equation}
Due to $H^n(E_i)<\ep^n$ in Lemma \ref{HnEiepn}, it follows that
\begin{equation}\aligned
\int_{\ep_iM}Ric_{\ep_iN}(\nu_i,\nu_i)\varphi^2\ge&\f{\k-(\k+c')\ep^2}{1+\ep^2}\int_{\f1{\de}+1}^{\f 1\de+\f2\ep-1}\bigg(-\ep^{1-n}H^n(E_i)\\
&+\int_{\ep_iM\cap(B^i_1\setminus B^i_{\ep})\times\{t\}}\sin^2\left(\pi\f{\log\r_i}{\log\ep}\right)\r_i^{1-n}\bigg)dt\\
\ge\f{\k-(\k+c')\ep^2}{1+\ep^2}&\int_{\f1{\de}+1}^{\f 1\de+\f2\ep-1}\bigg(\int_{\ep_iM\cap(B^i_1\setminus B^i_{\ep})\times\{t\}}\sin^2\left(\pi\f{\log\r_i}{\log\ep}\right)\r_i^{1-n}\bigg)dt\\
-2\f{1-\ep}{1+\ep^2}(\k-&(\k+c')\ep^2).
\endaligned
\end{equation}
Combining this with \eqref{vetau} we let $i\rightarrow\infty$ and obtain
\begin{equation}\aligned
&\f{\k-(\k+c')\ep^2}{1+\ep^2}\int_{\f1{\de}+1}^{\f 1\de+\f2\ep-1}\bigg(\int_{CY\cap(\mathcal{B}_1\setminus \mathcal{B}_{\ep})\times\{t\}}\sin^2\left(\pi\f{\log\r_\infty}{\log\ep}\right)\r_\infty^{1-n}\bigg)dt\\
&\qquad-2\f{1-\ep}{1+\ep^2}(\k-(\k+c')\ep^2)\\
\le&\int_{\f1{\de}}^{\f 1\de+\f2\ep}\left(\int_{CY\cap(\mathcal{B}_1\setminus \mathcal{B}_{\ep})\times\{t\}}\left(\f{3-n}2\sin\left(\pi\f{\log\r_\infty}{\log\ep}\right)
+\f{\pi}{\log\ep}\cos\left(\pi\f{\log\r_\infty}{\log\ep}\right)\right)^2\r_\infty^{1-n}\right)dt\\
&+\left(\int_{\f 1\de}^{\f 1\de+1}+\int^{\f1\de+\f 2\ep}_{\f 1\de+\f2\ep-1}\right)\left(\int_{CY\cap(\mathcal{B}_1\setminus \mathcal{B}_{\ep})\times\{t\}}\sin^2\left(\pi\f{\log\r_\infty}{\log\ep}\right)\r_\infty^{3-n}\right)dt,
\endaligned
\end{equation}
where $\r_{\infty}$ is the distance function on $\Si_{\infty}=CX$ from
the fixed point $o$ to the considered point. Letting
$\de\rightarrow0$, and using \eqref{TB2Bepde} we get
\begin{equation}\aligned
&\f{\k-(\k+c')\ep^2}{1+\ep^2}\int_{1}^{\f2\ep-1}\bigg(\int_{\p P\cap(\mathcal{B}_1\setminus \mathcal{B}_{\ep})}\sin^2\left(\pi\f{\log\r_\infty}{\log\ep}\right)\r_\infty^{1-n}\bigg)dt\\
&\qquad-2\f{1-\ep}{1+\ep^2}(\k-(\k+c')\ep^2)\\
\le&\int_{0}^{\f2\ep}\left(\int_{\p P\cap(\mathcal{B}_1\setminus \mathcal{B}_{\ep})}\left(\f{3-n}2\sin\left(\pi\f{\log\r_\infty}{\log\ep}\right)
+\f{\pi}{\log\ep}\cos\left(\pi\f{\log\r_\infty}{\log\ep}\right)\right)^2\r_\infty^{1-n}\right)dt\\
&+\left(\int_{0}^{1}+\int^{\f 2\ep}_{\f2\ep-1}\right)\left(\int_{\p P\cap(\mathcal{B}_1\setminus \mathcal{B}_{\ep})}\sin^2\left(\pi\f{\log\r_\infty}{\log\ep}\right)\r_\infty^{3-n}\right)dt.
\endaligned
\end{equation}

We calculate
\begin{equation}\aligned
\int_{\p P\cap(\mathcal{B}_1\setminus \mathcal{B}_{\ep})}\sin^2\left(\pi\f{\log\r_\infty}{\log\ep}\right)\r_\infty^{1-n}=&H^{n-2}(\p P\cap\p \mathcal{B}_1)\int_{\ep}^1\sin^2\left(\pi\f{\log s}{\log\ep}\right)\f1sds\\
=&\left(\log\f1\ep\right) H^{n-2}(\p P\cap\p \mathcal{B}_1)\int_0^1\sin^2(\pi t)dt.
\endaligned
\end{equation}
Hence we have
\begin{equation}\aligned
&\f{\k-(\k+c')\ep^2}{1+\ep^2}\left(\left(\f2\ep-2\right)\left(\log\f1\ep\right) H^{n-2}(\p P\cap\p \mathcal{B}_1)\int_0^1\sin^2(\pi t)dt-2(1-\ep)\right)\\
\le&\f2\ep\int_{\p P\cap(\mathcal{B}_1\setminus \mathcal{B}_{\ep})}\left(\f{3-n}2\sin\left(\pi\f{\log\r_{\infty}}{\log\ep}\right)+\f{\pi}{\log\ep}\cos\left(\pi\f{\log\r_{\infty}}{\log\ep}\right)\right)^2
\r_{\infty}^{1-n}\\
&+2\int_{\p P\cap(\mathcal{B}_1\setminus \mathcal{B}_{\ep})}\sin^2\left(\pi\f{\log\r_{\infty}}{\log\ep}\right)\r_{\infty}^{3-n}\\
\le&\f2\ep H^{n-2}(\p P\cap\p \mathcal{B}_1)\int_\ep^1\left(\f{3-n}2\sin\left(\pi\f{\log s}{\log\ep}\right)+\f{\pi}{\log\ep}\cos\left(\pi\f{\log s}{\log\ep}\right)\right)^2
\f1sds\\
&+2H^{n-2}(\p P\cap\p \mathcal{B}_1)\int_\ep^1\sin^2\left(\pi\f{\log s}{\log\ep}\right)sds\\
\le&2\left(\log\f1\ep\right) H^{n-2}(\p P\cap\p \mathcal{B}_1)\int_0^1\left(\f1\ep\left(\f{3-n}2\sin(\pi t)+\f{\pi}{\log\ep}\cos(\pi t)\right)^2+\sin^2(\pi t)\right)dt\\
=&2\left(\log\f1\ep\right) H^{n-2}(\p P\cap\p \mathcal{B}_1)\left(\f1\ep\left(\f{(n-3)^2}4+\f{\pi^2}{(\log\ep)^2}\right)+1\right)\int_0^1\sin^2(\pi t)dt.
\endaligned
\end{equation}
Together with \eqref{Hn-2pP} the above inequality implies
$$\k\le\f{(n-3)^2}4+o(\ep),$$
where $\lim_{\ep\rightarrow\infty}o(\ep)=0.$ Therefore we complete the proof.
\end{proof}
Finally, combining Theorem \ref{SplitLiouville} we obtain a Liouville theorem for minimal graphic functions without growth condition.
\begin{theorem}\label{BMGP}
Let $(\Si,\si)$ be a  complete $n-$dimensional Riemannian manifold
satisfying conditions C1), C2), C3) and with its non-radial Ricci
curvature satisfying $\inf_{\p
  B_\r}Ric_{\Si}\left(\xi^T,\xi^T\right)\ge\k\r^{-2}$ for some
constant $\k$ for sufficiently large $\r>0$, where $\xi$ is a local vector field on $\Si$ with $|\xi^T|=1$ defined in \eqref{Riccond}. If $\k>\f{(n-3)^2}4$, then any entire solution to \eqref{u} on $\Si$ must be a constant.
\end{theorem}
The number $\f{(n-3)^2}4$ in Theorem \ref{BMGP} is sharp, and we will construct examples to show this in the following section.

\section{Nontrivial entire minimal graphs in product manifolds}

Let $\Si$ be an Euclidean space $\R^{n+1}$ with a conformally flat  metric
$$ds_\phi^2=e^{\phi(r)}\sum_{i=1}^{n+1}dx_i^2,$$
where $r=|x|=\sqrt{x_1^2+\cdots+dx_{n+1}^2}$ and $\phi(|x|)$ is smooth in $\R^{n+1}$. Hence $\Si$ is a smooth manifold.
Set $\tilde{\phi}(r)=\int_0^re^{\f{\phi(r)}2}dr$. Let us define
$\r=\tilde{\phi}(r)$ and $\la(\r)=r\tilde{\phi}'(r)$, then the
Riemannian metric in $\Si$ can be written in  polar coordinates as
$$ds^2=d\r^2+\la^2(\r)d\th^2,$$
where $d\th^2$ is the  standard metric on $\S^{n}(1)$.
We assume $0\le\la'\le1$, $\la''\le0$,
\begin{equation}\aligned
\lim_{\r\rightarrow\infty}\f{\la(\r)}{\r}=\k,\qquad\lim_{\r\rightarrow\infty}\left(\r^2\f{1-(\la'(\r))^2}{\la^2(\r)}\right)=\f{1-\k^2}{\k^2},\qquad \lim_{\r\rightarrow\infty}\left(\r^2\f{\la''(\r)}{\la(\r)}\right)=0.
\endaligned
\end{equation}
From \cite{DJX}, there are examples satisfying the above conditions for every $\k\in(0,1]$.
Clearly, $\lim_{r\rightarrow\infty}\f1r\Si=CS_\k$ in the Gromov-Hausdorff measure, where $S_{\k}$ is an $n-$sphere in $\R^{n+1}$ with radius $0<\k\le1$, namely,
$$S_{\k}=\{(x_1,\cdots,x_{n+1})\in\R^{n+1}|\ x_1^2+\cdots+x_{n+1}^2=\k^2\}.$$
Moreover, let $\{e_\a\}_{\a=1}^{n}\bigcup\{\f{\p}{\p\r}\}$ be an orthonormal basis at the considered point of $\Si$. we calculate the sectional curvature and Ricci curvature of $\Si$ as follows (see Appendix A in \cite{Pl} for instance).
\begin{equation}\aligned\label{secRicSi}
K_{\Si}\left(\f{\p}{\p\r},e_\a\right)=-\f{\la''}{\la},\quad &K_{\Si}(e_\a,e_\be)=\f{1-(\la')^2}{\la^2},\\ Ric_{\Si}\left(\f{\p}{\p\r},e_\a\right)=0,\quad
&Ric_{\Si}\left(\f{\p}{\p\r},\f{\p}{\p\r}\right)=-n\f{\la''}{\la},\\
Ric_{\Si}(e_\a,e_\be)=\bigg((n-1)&\f{1-(\la')^2}{\la^2}-\f{\la''}{\la}\bigg)\de_{\a\be}.
\endaligned
\end{equation}
In particular, $Ric_\Si\ge0$ and $\lim_{\r\rightarrow\infty}\left(\r^2 Ric_{\Si}(e_\a,e_\be)\right)=\f{(n-2)^2}4\de_{\a\be}$ if $\k=\f{2}{n}\sqrt{n-1}$.

In theorem 3.4 of \cite{DJX}, we have showed that if $n\ge3$ and
$$\f{2}{n}\sqrt{n-1}\le\k<1,$$
then any hyperplane through the origin in $\Si$ is area-minimizing. Now we denote $T=\{(x_1,\cdots,x_{n+1})\in\R^{n+1}|\
x_{n+1}=0\}$ in $CS_{\k}$ or $\f1r\Si$ for $r>0$, and their induced metrics are determined by the ambient spaces. We will construct an entire minimal graph with non-constant graphic function in $\Si\times\R$ for every $\k\in[\f{2}{n}\sqrt{n-1},1)$, which obviously implies that the number $\f{(n-3)^2}4$ in Theorem \ref{BMGP} is sharp.

Let $D$ be the Levi-Civita connection of $\Si$. Let $\{E_i\}_{i=1}^{n+1}$ be the dual vectors of $\{dx_i\}_{i=1}^{n+1}$. Let $\G^k_{ij}$ be the Christoffel symbols of $\Si$ with respect to the frame $E_i$, i.e., $D_{E_i}E_j=\sum_k\G^k_{ij}E_k.$
Set $u^i=\si^{ij}u_j,\; |Du|^2=\si^{ij}u_iu_j,\;
D_iD_ju=u_{ij}-\G^k_{ij}u_k$ and $v=\sqrt{1+|Du|^2}$. We introduce an operator $\mathfrak{L}$ on a domain $\Om\subset\Si$ by
\begin{equation}\aligned\label{WL}
\mathfrak{L}F=\left(1+|DF|^2\right)^{\f32}\mathrm{div}_{\Si}\left(\f{D F}{\sqrt{1+|D F|^2}}\right)=\left(1+|D F|^2\right)\De_{\Si}F-F_{i,j}F^iF^j,
\endaligned
\end{equation}
where $F^i=\si^{ik}F_k$, and $F_{i,j}=F_{ij}-\G_{ij}^kF_k$ is the  covariant derivative.

Let $p=\f{n}2\k-\sqrt{\f{n^2\k^2}4-(n-1)}\ge\f1\k,$ then by Theorem 1.5 in \cite{Sp} there is a solution $u_j\in C^{\infty}(B_{j})$ to the Dirichlet problem
\begin{equation}\aligned\label{LD}
\left\{\begin{array}{cc}
     \mathfrak{L}u_j=0     & \quad\ \ \ {\rm{in}}  \  B_{j} \\ [3mm]
     u_j=c_jx_{n+1}r^{p-1}    & \quad\ \ \ {\rm{on}}  \  \p B_{j}
     \end{array}\right.,
\endaligned
\end{equation}
where $c_j$ is a positive constant and $r=\sqrt{x_1^2+\cdots+x_{n+1}^2}$. By symmetry,  $u_j(x',x_{n+1})+u_j(x',-x_{n+1})=0$ on $B_{j}$ with $x'=(x_1,\cdots,x_n)$. By \cite{DJX} and maximum principle, we have
\begin{equation}\aligned
|u_j|\ge c_j|x_{n+1}|r^{p-1} \ \ {\rm{in}}  \  B_{j}.
\endaligned
\end{equation}
If $w_j$ is a solution of \eqref{LD} with boundary $d_jx_{n+1}r^{p-1}$ and $0<d_j<c_j$, we have $|u_j|>|w_j|$ on $B_j\cap\{x_{n+1}\neq0\}$. By the uniqueness of the solution of \eqref{LD} there is a $c_j>0$ such that $\sup_{B_1}|Du_j|=1$.

Let $\G_j(s)=\{x\in B_j|\ u_j(x)=s\}$. We claim $u_j(x',t)\ge u_j(x',s)$ for all $(x',t),(x',s)\in B_j$ and $t>s$. If not, without loss of generality there are $t_1<t_2<t_3$ so that $u_j(x',t_1)=u_j(x',t_2)=\tau_j<u_j(x',t_3)$. It is not hard to see that there is a closed curve $\vartheta_{\tau_j}\subset\G_j(\tau_j)$ in the half plane $\{(s,x_{n+1})\in\R^2|\ s=|x'|\ge0\}$ such that $u_j(z)>\tau_j$ for every $z\in U_j$ and $U_j$ is a domain in $\{(s,x_{n+1})\in\R^2|\ s=|x'|\ge0\}$, which is enclosed by $\vartheta_{\tau_j}$. By the symmetry of $u_j$, rotating $\vartheta_{\tau_j}$ on $x'=(x_1,\cdots,x_n)$ generates an $n$-dimensional set $\widetilde{\G_j}(\tau_j)\subset\G_j(\tau_j)$, which encloses a domain $\widetilde{U}_j\subset B_j$. Then $u_j(x)>\tau_j$ for every $x\in \widetilde{U}_j$. But this is impossible as $\mathrm{graph}_{u_j}$ is area-minimizing. Hence $\f{\p u_j}{\p x_{n+1}}\ge0$.

There is a subsequence $\{j'\}$ of $\{j\}$ such that $u_{j'}$
converges to a function $u$ defined on $\Om\subset\Si$ by varifold
convergence, where $\mathrm{graph}_u$ is also area-minimizing. By the
symmetry of $u$, we deduce that $\Om$ is symmetric with respect to
$x_1,x_2,\cdots,x_n$, and $\Om$ is symmetric with respect to $x_{n+1}$. Clearly, $u$ is smooth, $\mathfrak{L}u=0$ and $\lim_{x\rightarrow\p\Om^\pm\setminus T}u(x)=\pm\infty$, where $\Om^{\pm}=\Om\cap\{\pm x_{n=1}>0\}$. In particular, $\Om^-=\{(x',-x_{n+1})|\ (x',x_{n+1})\in\Om^+\}$, $u(x',x_{n+1})+u(x',-x_{n+1})=0$, $Du(0)=0$ and $\sup_{B_1}|Du|=1$. Moreover, $u(x',t)\ge u(x',s)$ for all $(x',t),(x',s)\in \Om$ and $t>s$.

We want to show $\Om=\Si$. If not, $\p\Om\neq\emptyset$ and both of
$\p\Om^\pm$ are area-minimizing hypersurfaces in $\Si$. Since also
$\p\Om^+$ is symmetric with respect to $x_1,x_2,\cdots,x_n$, then
$\p\Om^+$ must be a graph on a domain of the half sphere, or else we
cannot have an  area-minimizing hypersurface $\p\Om^+$. If $w$ is the graphic function of $\p\Om^+$ on a closed half of unit sphere $\S^n$, then $w$ satisfies (see formula (2.6) in \cite{D1} for instance)
\begin{equation}\aligned\label{GraphS}
\De_{\S}w-2\f{\la'(w)}{\la(w)}|\na_{\S}w|^2-n\la(w)\la'(w)=0,
\endaligned
\end{equation}
where $\De_\S$ and $\na_\S$ are Laplacian and Levi-Civita connection
of $\S^n$, respectively. By the definition of $\la$, the equation \eqref{GraphS} has a unique
smooth solution for the Dirichlet problem with fixed boundary if $w>0$. So $\p\Om^+$ is a smooth hypersurface in $\Si$. $\f1r\p\Om^+$ is also an area-minimizing hypersurface in $\f1r\Si$, then $\lim_{r\rightarrow\infty}\f1r\p\Om^+$ will converge to an area-minimizing cone $CX$ over $X$ in $CS_\k$, where $X$ is contained in a closed half  sphere. Hence $X$ must be an equator and $CX$ is just a hyperplane $T$.

Let $S=\p\Om^+$. The second variation formula implies that there is a Jacobi field operator $L_S$ given by
\begin{equation}\aligned
L_Sh=\De_Sh+\left(|B|^2+Ric_{\Si}(\nu,\nu)\right)h,
\endaligned
\end{equation}
where $\De_S$, $B$ are the Laplacian and second fundamental form for $S$ relative to the metric $\si$, and $Ric_{\Si}$ is the Ricci curvature of $\Si$ relative to $\si$. Let $\S^{n-1}$ be an $(n-1)$-dimensional unit sphere in $\R^n$. Since $\f1tS$ converges to $T=\R_+\times_{\k\r}\S^{n-1}$ as $t\rightarrow\infty$, in terms of the coordinates $(\r,\a)\in(0,\infty)\times\S^{n-1}$ we consider
\begin{equation}\aligned\label{LTh}
L_Th=&\De_Th+\f{n-1}{\r^2}\left(\f1{\k^2}-1\right)h\\
=&\f1{\r^{n-1}}\f{\p}{\p \r}\left(\r^{n-1}\f{\p h}{\p \r}\right)+\f1{\k^2\r^2}\De_{\S^{n-1}}h+\f{n-1}{\r^2}\left(\f1{\k^2}-1\right)h\\
=&\f1{\r^2}\left(\r^2\f{\p^2}{\p \r^2}+(n-1)\r\f{\p}{\p \r}+\f1{\k^2}\De_{\S^{n-1}}+(n-1)\left(\f1{\k^2}-1\right)\right)h.
\endaligned
\end{equation}
The only positive solutions of $L_Th=0$ on $T$ are
\begin{equation}\aligned
h=\bar{c}_1\r^{-\la_-}+\bar{c}_2\r^{-\la_+},
\endaligned
\end{equation}
where $\bar{c}_1,\bar{c}_2$ are constants and $\la_{\pm}=\f{n-2}2\pm\sqrt{\f{n^2}4-\f{n-1}{\k^2}}>0$. By \cite{CHS,HS}, the equation $L_Tw=f_0$ with $|f_0|\le c\r^{-2-\la_--\de}$, $\de>0$ has a nonnegative solution
\begin{equation}\aligned\label{asymw}
w=(\bar{c}_3+\bar{c}_4\log \r)\r^{-\la_-}+O(\r^{-\la_--\de'})(\de'>0)\qquad \mathrm{as}\ \r\rightarrow\infty,
\endaligned
\end{equation}
where $\bar{c}_3,\bar{c}_4$ are constants, and $\bar{c}_4=0$ in case $\la_+>\la_-$, namely, $\k>\f2n\sqrt{n-1}$.

Now we see $S$ as a graph $\{(x',f(x'))|\ x'\in T\setminus K\}$ outside some bounded domain in $(\R^{n+1},ds^2_\phi)$, where $x'=(x_1,\cdots,x_n)$, and
$$f(x')=\int_0^{h(x')}e^{-\f12\phi(\sqrt{|x'|^2+s^2})}ds$$
with $|x'|^2=x_1^2+\cdots+x_n^2$.
Let $\g_{\r\a}$ be a unit normal smooth line in the upper plane of $\Si$ which corresponds to $\{(x',x_{n+1})\in\R^{n+1}|\ x_{n+1}\ge0,\ x'=x_\a\}\subset(\R^{n+1},ds^2_\phi)$ with $x_\a$ being $\r\a$ in the polar coordinate and $\g_{\r\a}(0)\in T$.
Now $S=\{\g_{\r\a}(w(\r\a))|\ \r\a\in T\setminus K\}$, where
$w(\r\a)=h(x')$ if $x'$ is $\r\a$ in the polar coordinate. We can
embed isometrically $\Si$ into $(n+2)$-dimensional Euclidean space,
then $S$ is an $n$-dimensional submanifold in $\R^{n+2}$ with mean
curvature decaying as $O(\f1r)$. Hence by the Allard regularity theorem (see
\cite{Al} or \cite{S}), for any $\ep>0$ there is a $\r_0>0$ such that
$$\r^{-1}|w(\r\a)|+|\na_Tw(\r\a)|\le \ep\qquad \mathrm{for\ every}\ \r\ge \r_0,\ \a\in\S^{n-1},$$
where $\na_{T}$ is the Levi-Civita connection of $T$ with induced metric in $CS_\k$.
Since $S$ is a graph with graphic function $f$ outside a compact set in $(\R^{n+1},ds_\phi^2)$, 
then the mean curvature is $H=\f n2\f{\phi'(r_f)}{r_f}X^N$ with $r_f=\sqrt{|x'|^2+f^2(x')}$, and $(\cdots)^N$ is the projection onto the normal bundle $NS$, namely, $f$ satisfies (after a simple computation)
\begin{equation}\aligned
\sum_{i,j=1}^ng^{ij}f_{ij}=\f n2\f{\phi'}{r_f}\left(-\sum_{i=1}^nx_if_i+f\right),
\endaligned
\end{equation}
where $g_{ij}=\de_{ij}+f_if_j$ and $(g^{ij})$ is the inverse matrix of $(g_{ij})$.
Then by the estimates \eqref{asymw} and the Schauder estimates we obtain
\begin{equation}\aligned
\tilde{r}^{-1}|f(x')|+|\na_{\R}f(x')|+\tilde{r}|\na_{\R}^2f(x')|\le C\tilde{r}^{-\ep}
\endaligned
\end{equation}
with $\tilde{r}=\sqrt{x_1^2+\cdots+x_n^2}$ and some $\ep,C>0$, where $\na_{\R}$ is the standard Levi-Civita connection of Euclidean space. 
By the definition of $w$ and applying a coordinate transformation,  we deduce
\begin{equation}\aligned\label{wDw}
\r^{-1}|w(\r\a)|+|\na_Tw(\r\a)|+\r|\na_T^2w(\r\a)|\le C'\r^{-\ep'}
\endaligned
\end{equation}
for some $\ep',C'>0$. In particular, $w$ satisfies \eqref{asymw}.

The $\mathrm{graph}_u$ has a unique tangent cone at infinity: a cylinder $T\times\R$. Let $\G_y=\G_y(u)=\{x\in\Si|\ u(x)=y\}$ for any $y\in\R$. Since $\lim_{x\rightarrow\p\Om^\pm\setminus T}u(x)=\pm\infty$ and $u(x',t)\ge u(x',s)$ for each $(x',t),(x',s)\in \Om$ and $t>s$, we deduce $\mathrm{dist}(\G_{y_2},0)\ge\mathrm{dist}(\G_{y_1},0)>0$ for $|y_2|\ge|y_1|$ and sufficiently large $|y_1|$. Moreover, for any $\ep>0$ there is $y_0=y_0(\ep)$ such that $\G_y$ is within $\ep$ of $S$ for any $y\ge y_0$. Now we use $\r\omega$ to represent the element of $\Si$ with metric $\si=d\r^2+\la^2(\r)d\th^2$, and claim that
\begin{equation}\aligned
|Du(\r\omega)|\rightarrow\infty  \qquad \mathrm{as}\qquad  |u(\r\omega)|+\r\rightarrow\infty,\  (\r,\omega)\in\left((0,\infty)\times_\la\S^n\right)\cap\Om.
\endaligned
\end{equation}
If we embed $\Si$ into $\R^{n+2}$ isometrically, then
$\mathrm{graph}_u$ is an $(n+1)$-dimensional submanifold in $\R^{n+3}$
with codimension 2. We check  that the Allard regularity theorem still works in our case. Invoking  elliptic
regularity theory, if the minimal hypersurfaces $M_k$ converge to a
cone $C$ in varifold sense, then the convergence is $C^2$ near regular
points of $C$. For any $\mu_k,\la_k\rightarrow\infty$ it is clear that
$\mathrm{graph}_{u-\mu_k}=\{(x,u(x)-\mu_k)|\ x\in\Si\}$ converges to
$S\times\R\subset\Si\times\R$ in the varifold sense and
$\f1{\la_k}\left(\mathrm{graph}_{u-\mu_k}\right)$  converges to
$T\times\R\subset CS_\k\times\R$ in the varifold sense.
So we can show the above claim by $C^2$ convergence (see also the proof of Theorem 4 in \cite{Si2}).

Denote $\widetilde{T}=T\times\R$. In terms of the coordinates
$(\r,\a,y)\in(0,\infty)\times\S^{n-1}\times\R$ we can write the operator $L_{\widetilde{T}}$ as
\begin{equation}\aligned
L_{\widetilde{T}}h=&\De_Th+\f{\p^2h}{\p y^2}+\f{n-1}{\r^2}\left(\f1{\k^2}-1\right)h\\
=&\f1{\r^{n-1}}\f{\p}{\p \r}\left(\r^{n-1}\f{\p h}{\p \r}\right)+\f{\p^2h}{\p y^2}+\f1{\k^2\r^2}\De_{\S^{n-1}}h+\f{n-1}{\r^2}\left(\f1{\k^2}-1\right)h.\\
\endaligned
\end{equation}
Let
$$v(\r,y)=\r^{\la_-}\int_{\S^{n-1}}h(\r\a,y)d\a\qquad \mathrm{and}\qquad f(\r,y)=\r^{\la_-}\int_{\S^{n-1}}\tilde{f}(\r\a,y)d\a$$
with $\la_{-}=\f{n-2}2-\sqrt{\f{n^2}4-\f{n-1}{\k^2}}>0$, then $L_{\widetilde{T}}=\tilde{f}$ implies that
\begin{equation}\aligned
\r^{-1-\be}\f{\p}{\p \r}\left(\r^{1+\be}\f{\p v}{\p \r}\right)+\f{\p^2v}{\p y^2}=f
\endaligned
\end{equation}
with $\be=2\sqrt{\f{n^2}4-\f{n-1}{\k^2}}$. The left of the above equation is a uniform elliptic operator for $\r\ge c$ with any positive constant $c$.

By the Allard regularity theorem, there are a constant $\r_1>0$ and a domain
$$G=\{(\r\a,y)|\ \r\ge \r_1,\ y\in\R,\ \r\a\in T\}\subset T\times\R=\widetilde{T}$$
such that $\mathrm{graph}_u$ can be written as a graph in $G$ with graphical function $W$ outside some compact set $\widetilde{K}$ in $\Si\times\R$. Namely, $\mathrm{graph}_u\setminus\widetilde{K}=\mathrm{graph}_W=\{(\g_{\r\a}(W(\r\a,y)),y)|\ \r\a\in T,\ \r\ge\r_1,\ y\in\R\}$ and $\g_{\r\a}$ is defined as before. Similar to \eqref{wDw} we have
\begin{equation}\aligned\label{uDu}
\r^{-1}|W(\r\a)|+|\na_{\widetilde{T}}W(\r\a)|+\r|\na_{\widetilde{T}}^2W(\r\a)|\le C\r^{-\de}
\endaligned
\end{equation}
for some $\de,C>0$, where $\na_{\widetilde{T}}$ is the Levi-Civita connection of $\widetilde{T}$ with induced metric in $CS_\k\times\R$. Then by Theorem 1 in \cite{Si2}, we obtain for any $\ep\in(0,1)$
\begin{equation}\aligned
|y|^\ep \r^{\la_-}\f{\p W}{\p y}(\r\a,y)\ge C_2 \quad \mathrm{for\ all}\ y\in\R,\ \a\in\S^{n-1},\ y\ge \r\ge C_1,
\endaligned
\end{equation}
where $C_1,C_2$ are constants independent of $y$ and $\r$. It is clear that
\begin{equation}\aligned
\f{\p W}{\p y}(\r\a,y)=\f1{|Du(\xi)|}
\endaligned
\end{equation}
where $(\r\a,y)\in G$, $y=u(\xi)$ and $\xi=\g_{\r\a}(W(\r\a,y))$. Fix $\r$, we have $|Du(\xi)|\le C_3|u(\xi)|^\ep$ with constant $C_3$ depending on $\r$. Hence $|D(u(\xi))^{1-\ep})|$ is bounded when $\xi$ approaches $S$ in $\g_{\r\a}$, which contradicts  $\lim_{x\rightarrow\p\Om^\pm\setminus T}u(x)=\pm\infty$. Therefore, we deduce $\Om=\Si$, namely, we get a smooth entire minimal graph $\{(x,u(x))|\ x\in\Si\}$ on $\Si$.

\begin{theorem}
Let $\Si$ be an $(n+1)$-dimensional Riemannian manifold described in the front of this section. If $n\ge3$ and
$\f{2}{n}\sqrt{n-1}\le\k<1,$ then there exists a smooth entire minimal graph $\{(x,u(x))|\ x\in\Si\}$ in $\Si\times\R$, where $u$ is not a constant. 
\end{theorem}

\bigskip

\bigskip

\bibliographystyle{amsplain}

\end{document}